\documentclass[11pt]{amsart}

\begin{document}
\newtheorem{theo}{Theorem}[section]
\newtheorem{prop}[theo]{Proposition}
\newtheorem{lemme}[theo]{Lemma}
\newtheorem{cor}[theo]{Corollary}
\newtheorem{rem}[theo]{Remark}
\newtheorem{conj}[theo]{Conjecture}
\newtheorem{fact}[theo]{Fact}

\numberwithin{equation}{section}

\leftskip -1cm
\rightskip -1cm

\def\aa{{\mathbb A}}
\def\bc{\backslash}
\def\lg{L^2(Z(\aa)G(F)\bc G(\aa))}
\def\o{\omega}
\def\lgg{L^2(Z(\aa)G'(F)\bc G'(\aa);\o)}
\def\lra{\leftrightarrow}
\def\ccc{{\bf C}}
\def\jlr{{\bf JL}_r}
\def\ljr{{\bf LJ}_r}

\def\s{{\frak S}}
\numberwithin{equation}{section}
\leftskip -1cm
\rightskip -1cm
\def\a{\alpha}
\def\b{{\mathcal B}}
\def\d{{\mathcal D}}
\def\e{\varepsilon}
\def\f{{\mathcal F}}
\def\i{{\bf i}}
\def\k{\{1,2,...,k\}}
\def\l{{\bf l}}
\def\n{\mathbb N}
\def\r{{\mathbb R}}
\def\s{\sigma}
\def\z{\mathbb Z}

\def\cc{{\mathbb C}}
\def\ccc{{\bf C}}
\def\lra{\leftrightarrow}
\def\ski{{\sum_{i=1}^k}}
\def\ki{{_{i=1}^k}}
\def\rrr{{\mathcal R}}

\def\lj{{\bf LJ}}
\def\sgls{standard Levi subgroup}
\def\sglss{standard Levi subgroups}
\def\bc{\backslash}
\def\cusp{{\mathcal C}}
\def\tr{{\rm tr}}

\def\Ind{\mathrm{Ind}}

\title[Global Jacquet-Langlands]
{Global Jacquet-Langlands correspondence, multiplicity one and
classification of automorphic representations} \maketitle
\centerline{\footnotesize{by Alexandru Ioan
BADULESCU}{\footnote{Alexandru Ioan BADULESCU, Universit\'e de
Poitiers, UFR Sciences SP2MI, D\'epartement de Math\'ematiques,
T\'el\'eport 2, Boulevard Marie et Pierre Curie, BP 30179, 86962
FUTUROSCOPE CHASSENEUIL CEDEX\\
E-mail : badulesc\makeatletter @\makeatother
math.univ-poitiers.fr}}} \centerline{\footnotesize{with an Appendix
by Neven GRBAC}{\footnote{Neven GRBAC, University of Zagreb, Department of Mathematics,
Unska 3, 10000 Zagreb, Croatia\\
E-mail : neven.grbac\makeatletter @\makeatother zpm.fer.hr}}}
\ \\

{\bf Abstract:} In this paper we show a local Jacquet-Langlands
correspondence for all unitary irreducible representations. We prove
the global Jacquet-Langlands correspondence in characteristic zero.
As consequences we obtain the multiplicity one and strong
multiplicity one theorems for inner forms of $GL(n)$ as well as a
classification of the residual spectrum and automorphic
representations in analogy with results proved by
Moeglin-Waldspurger and Jacquet-Shalika for $GL(n)$.

\tableofcontents

\section{Introduction}

The aim of this paper is to prove the global Jacquet-Langlands
correspondence and its consequences for the theory of
representations of the inner forms of $GL_n$ over a global field of
characteristic zero. In order to define the global Jacquet-Langlands
correspondence, it is not sufficient to transfer only square
integrable representations as in the classical theory. It would be
necessary to transfer at least the local components of global
discrete series. Here we prove, more generally, the transfer of all
unitary representations. Then we prove the global Jacquet-Langlands
correspondence, which is compatible with this local transfer. As
consequences we obtain for inner forms of $GL_n$ the multiplicity
one theorem and strong multiplicity one theorem, as well as a
classification of the residual spectrum \`a la Moeglin-Waldspurger
and unicity of the cuspidal support \`a la Jacquet-Shalika. Using
these classifications we give counterexamples showing that the
global Jacquet-Langlands correspondence for discrete series do not
extend well to all unitary automorphic representations.

We give here a list of the most important results, starting with the
local study. We would like to point out that most of the local
results in this paper were obtained by Tadi\'c in [Ta6] under the
assumption that his conjecture $U_0$ holds. After we proved these
results here independently of his conjecture, S\'echerre announced
the proof of the conjecture $U_0$. His methods are different, based
on the theory of types and also some works of Barbash-Moy. The
proofs here are based on the Aubert (-Zelevinsky-Schneider-Stuhler)
involution and an irreducibility trick.

Let $F$ be a local non-Archimedean
field of characteristic zero and $D$ a central division algebra over $F$ of dimension $d^2$.
For $n\in\n^*$ set $G_n=GL_n(F)$ and $G'_n=GL_n(D)$. Let $\nu$ generically denote the character given by the
absolute value of the reduced norm on groups like $G_n$ or $G'_n$.

Let $\s'$ be a square integrable
 representation of $G'_n$. If $\s'$ is a cuspidal representation, then it corresponds
by local Jacquet-Langlands to a square integrable representation $\s$ of $G_{nd}$. We set $s(\s')=k$, where $k$ is
the length of the Zelevinsky segment of $\s$. If $\s'$ is not cuspidal, we set $s(\s')=s(\rho)$, where $\rho$ is any
cuspidal representation in the cuspidal support of $\s'$, and this does not depend on the choice. We set then $\nu_{\s'}=
\nu^{s(\s')}$. For any $k\in\n^*$ we denote then
$u'(\s',k)$ the Langlands quotient of the induced representation from
$\otimes_{i=0}^{k-1}(\nu_{\s'}^{\frac{k-1}{2}-i}\s')$,
and if $\a\in ]0,\frac{1}{2}[$, we denote
$\pi'(u'(\s',k),\a)$ the induced representation from $\nu_{\s'}^\a u'(\s',k)\otimes \nu_{\s'}^{-\a} u'(\s',k)$.
The representation $\pi'(u'(\s',k),\a)$  is irreducible ([Ta2]). Let $\mathcal U'$ be the set of all
representations of type $u'(\s',k)$ or
$\pi'(u'(\s',k),\a)$ for all $G'_n$, $n\in \n^*$. Tadi\'c conjectured in [Ta2] that

(i) all the representations in $\mathcal U'$ are unitary;

(ii) an induced representation
from a product of representations in $\mathcal U'$ is always irreducible and unitary;

(iii) every irreducible unitary representation of
$G'_m$, $m\in \n^*$,
is an induced representation from a product of representations in $\mathcal U'$.\\

The fact that the $u'(\s',k)$ are unitary has been proved in [BR1]
if the characteristic of the base field is zero. In the third
section of this paper we complete the proof of the point (i) (i.e.
$\pi'(u'(\s',k),\a)$ are unitary; see corollary  \ref{next}) and
prove (ii) (proposition  \ref{unit}).

We also prove a Jacquet-Langlands transfer for all irreducible
unitary representations of $G_{nd}$. More precisely, let us write $g'\lra g$ if $g\in G_{nd}$, $g'\in G'_n$ and
the characteristic polynomials of $g$ and $g'$ are equal and have distinct roots in an algebraic closure of $F$.
Denote $G_{nd,d}$ the set of elements $g\in G_{nd}$ such that there exists $g'\in G'_n$ with $g'\lra g$.
We denote $\chi_\pi$ the function character of an admissible representation $\pi$.
We say a representation $\pi$ of $G_{nd}$ is $d${\bf -compatible} if there exists
$g\in G_{nd,d}$ such that $\chi_\pi(g)\neq 0$.
We have (proposition  \ref{unit}):\\
\ \\
{\bf Theorem.} {\it If $u$ is a $d$-compatible irreducible unitary
representation of $G_n$, then there exists a unique irreducible
unitary representation $u'$ of $G'_n$ and a unique sign
$\varepsilon\in \{-1,1\}$ such that
$$\chi_u(g)=\varepsilon \chi_{u'}(g')$$
for all $g\in G_{nd,d}$ and $g'\lra g$.}\\
\ \\
It is Tadi\'c the first to point out ([Ta6]) that this should hold
if his conjecture $U_0$ were true. For more precise formulas for the
transfer which are essential for the global study see the subsection
\ref{newform}.
\\

The fifth section contains global results. Let us use the theorem above to define a map $|\lj|:u\mapsto u'$ from the
set of irreducible unitary $d$-compatible representations of $G_{nd}$ to the set of irreducible unitary representations
of $G'_n$.

Let now $F$ be a global field {\it of characteristic zero} and $D$ a
central division algebra over $F$ of dimension $d^2$. Let
$n\in\n^*$. Set $A=M_{n}(D)$. For each place $v$ of $F$ let $F_v$ be
the completion of $F$ at $v$ and set $A_v=A\otimes F_v$. For every
place $v$ of $F$, $A_v\simeq M_{r_v}(D_v)$ for some positive number
$r_v$ and some central division algebra $D_v$ of dimension $d_v^2$
over $F_v$ such that $r_v d_v=nd$. We will fix once and for all an
isomorphism and identify these two algebras. We say that $M_n(D)$ is
split at the place $v$ if $d_v=1$. The set $V$ of places where
$M_n(D)$ is not split is finite. We assume in the sequel that $V$
does not contain any infinite place.

Let $G_{nd}(\aa)$ be the group of ad\`eles of $GL_{nd}(F)$,
  and $G'_n(\aa)$
the group of ad\`eles of $GL_n(D)$. We identify
$G_{nd}(\aa)$ with $M_{nd}(\aa)^\times$ and $G'_{n}(\aa)$ with $A(\aa)^\times$.

Denote $DS_{nd}$ (resp. $DS'_n$) the set of discrete series of $G_{nd}(\aa)$ (resp. $G'_n(\aa)$).
If $\pi$ is a discrete series of $G_{nd}(\aa)$ or $G'_n(\aa)$, if $v$ is a place of $F$, we denote $\pi_v$ the
local component of $\pi$ at the place $v$. We will say that a discrete series $\pi$ of $G_{nd}(\aa)$ is
$D${\bf -compatible} if $\pi_v$ is $d_v$-compatible for all place $v\in V$.

If $v\in V$, the Jacquet-Langlands correspondence for
$d_v$-compatible unitary representations between $GL_{nd}(F_v)$ and
$GL_{r_v}(D_v)$ will be denoted $|\lj|_v$. Recall that if $v\notin
V$, we have identified the groups $GL_{r_v}(D_v)$ and
$GL_{nd}(F_v)$.
We have the following (theo. \ref{correspondence}):\\
\ \\
{\bf Theorem.} (a) {\it There exists a unique injective map ${\bf G}:DS'_n\to DS_{nd}$ such that, for all $\pi'\in DS'_n$,
we have ${\bf G}(\pi')_v=\pi'_v$ for every place
$v\notin V$. For every $v\in V$,
${\bf G}(\pi')_v$ is $d_v$-compatible and
we have $|\lj|_v({\bf G}(\pi')_v)=\pi'_v$. The image of $\bf G$ is the set of $D$-compatible elements of $DS_{nd}$.}

(b) {\it One has multiplicity one and strong multiplicity one theorems
for the discrete spectrum of $G'_n(\aa)$.}\\

Global correspondences with division algebras under some conditions
(on the division algebra or on the representation to be transferred)
have already been carried out (at least) in [JL], [Ro], [DKV] and
[Vi]. The general result here is heavily based on the comparison of
the trace formulas for $G'_n(\aa)$ and
$G_{nd}(\aa)$ carried out in [AC].\\

In the sequel of the fifth section we give a classification of
representations of $G'_n(\aa)$. We define the notion of basic
cuspidal representation for groups of type $G'_k(\aa)$ (see
proposition  \ref{cuspidal} and the sequel). These basic cuspidal
representations are all cuspidal. Neven Grbac will show in his
Appendix that these are actually the only cuspidal representations.
Then residual discrete series of $G'_n(\aa)$ are obtained from
cuspidal representations in the same way residual discrete series of
$GL_n(\aa)$ are obtained from cuspidal representations in [MW2].

Moreover, for any (irreducible) automorphic representation $\pi'$
of $G'_n$, we know that ([La]) there exists a
couple $(P',\rho')$ where $P'$ is a parabolic subgroup of $G'_n$ containing
the group of upper triangular matrices and
$\rho'$ is a cuspidal representation of the Levi factor $L'$ of $P'$
 twisted by a real non ramified character such that
$\pi'$ is a constituent (in the sense of [La]) of the induced
representation from $\rho'$ to $G'_n$ with respect to $P'$. We prove
(proposition \ref{classif} (c)) that this couple $(\rho',L')$ is
unique up to conjugation. This result is an analogue for $G'_n$ of
the theorem 4.4 of [JS].

The last section is devoted to the computation of $L$-functions,
$\epsilon'$-factors (in the meaning of [GJ]) and their behavior
under local transfer of irreducible (especially unitary)
representations. The behavior of the $\epsilon$-factors then
follows. These calculations are either well known or trivial, but we
feel it is natural to give them explicitly here. The $L$-functions
and $\epsilon'$-factors in question are preserved under the
correspondence for square integrable representations. In general,
$\epsilon'$-factors (but not $L$-functions) are preserved under the
correspondence for irreducible unitary representations.

In the Appendix Neven Grbac completes the classification of the
residual spectrum by showing that some representations are residual.

The essential part of this work has been done at the Institute for
Advanced Study during the year 2004 and I want to thank the
Institute for their warm hospitality and support. This research has
been supported by the NSF fellowship no. DMS-0111298. I want to
thank Robert Langlands and James Arthur for useful discussions about
global representations; Marko Tadi\'c and David Renard for useful
discussions on the local unitary dual; Abderrazak Bouaziz who
explained to me the intertwining operators. I want to thank Guy
Henniart and Colette Moeglin for the interest they showed for this
work and their invaluable advices. I thank Neven Grbac for his
Appendix where he carries out a last and important step of the
classification. Discussions with Neven Grbac have been held during
our stay at the Erwin Schr\"{o}dinger Institute in Vienna and I want
to thank here Joachim Schwermer for his invitation.

\section{Basic facts and notations (local)}

Let
$F$ be a non-Archimedean local field and
$D$ a central division algebra of finite dimension  over $F$. Then the dimension of $D$ over $F$ is
a square $d^2$, $d\in \n^*$.
If $n\in \n^*$, we set $G_n=GL_n(F)$ and $G'_n=GL_n(D)$.  From now on we identify a smooth representation
of finite length with
its equivalence class, so we will consider two equivalent representations as being equal. By {\bf character} of
$G_n$  we
mean a smooth representation of dimension one of $G_n$. In particular a character is not unitary unless we specify it.
Let $\s$ be an irreducible
smooth representation of
$G_n$. We say $\s$ is {\bf square integrable} if $\s$ is unitary and has a non-zero coefficient which is square integrable
modulo the center of $G_n$. We say $\s$ is {\bf essentially square integrable} if $\s$ is the twist of a square integrable
representation by a character of $G_n$. We say $\s$ is {\bf cuspidal} if $\s$ has a non-zero coefficient which has
compact support modulo the center of $G_n$. In particular a cuspidal representation is essentially square integrable.\\
\ \\
For all $n\in \n^*$ let us fix the following notations:

$Irr_n$ is the set of smooth irreducible representations of $G_n$,

$\d_n$ is the subset of essentially square integrable representations in $Irr_n$,

$\cusp_n$ is the subset of cuspidal representations in $\d_n$,

$Irr^u_n$ (resp. $\d^u_n$, $\cusp^u_n$) is the subset of unitary representations in $Irr_n$ (resp. $\d_n$, $\cusp_n$),

$\rrr_n$ is the Grothendieck group of admissible representations of finite length of $G_n$,

$\nu$ is the character of $G_n$ defined by the absolute value of the determinant (notation independent of $n$ --
this will lighten the notations and cause no ambiguity in the sequel).

For any $\s\in \d_n$, there is a unique couple $(e(\s),\s^u)$ such that $e(\s)\in \r$, $\s^u\in \d^u_n$ and
$\s=\nu^{e(\s)}\s^u$.

We will systematically  identify  $\pi\in Irr_n$ with its image in $\rrr_n$ and consider $Irr_n$ as a subset of
$\rrr_n$. Then $Irr_n$ is a $\z$-basis of the $\z$-module $\rrr_n$.

If $n\in \n^*$ and $(n_1,n_2,...,n_k)$ is an ordered set of positive integers such that $n=\sum_{i=1}^kn_i$
then the subgroup $L$ of $G_n$ made of diagonal matrices by blocs of sizes $n_1,n_2,...,n_k$ in this order from the
left up corner
to the right down corner is called a {\bf \sgls} of $G_n$. The group $L$ is canonically
isomorphic
with the product $\times_{i=1}^kG_{n_i}$, and we will identify these two groups. Then the notations
$Irr(L)$, $\d(L)$, $\cusp(L)$, $\d^u(L)$,
$\cusp^u(L)$, $\rrr(L)$ extend in an obvious way to $L$.
In particular $Irr(L)$ is canonically isomorphic to $\times_{i=1}^k Irr_{n_i}$ and so on.

We denote $ind_L^{G_n}$ the normalized parabolic induction functor
where it is understood that we induce with respect to the parabolic
subgroup of $G_n$ containing $L$ and the subgroup of upper
triangular matrices. Then $ind_L^{G_n}$ extends to a group morphism
$\i_L^{G_n}:{\mathcal R}(L)\to {\mathcal R}_n$. If $\pi_i\in
\rrr_{n_i}$ for $i\in\{1,2,...,k\}$ and $n=\ski n_i$, we denote
$\pi_1\times\pi_2\times...\times \pi_k$ or abridged
$\prod_{i=1}^k\pi_i$ the representation
$$ind_{\times_{i=1}^kG_{k_i}}^{G_n}\otimes\ki \s_i$$
of $G_n$. Let $\pi$ be a smooth representation of finite
length of $G_n$. If distinction between quotient, subrepresentation and subquotient of $\pi$ is not relevant,
we consider $\pi$ as an element of $\rrr_n$ (identification with its class) with no extra explanation.

If $g\in G_n$ for some $n$, we say $g$ is {\bf regular semisimple} if the characteristic polynomial of
$g$ has distinct roots in an algebraic closure of $F$.
If $\pi\in \rrr_n$, then we let $\chi_\pi$ denote the function character of $\pi$, as a locally constant map, stable under
conjugation, defined on the set
of regular semisimple elements of $G_n$.

We adopt the same notations adding a sign $'$
for $G'_n$: $Irr'_n$, $\d'_n$, $\cusp'_n$, $Irr_n^{'u}$, $\d_n^{'u}$, $\cusp_n^{'u}$, $\rrr'_n$.

There is a standard way of defining the determinant and the characteristic polynomial for elements of $G'_n$, in spite
$D$ is non commutative
(see for example [Pi] section  14). If $g\in G'_n$, then the characteristic polynomial of $g$ has coefficients in $F$,
is monic and has degree $nd$.
The definition of a regular semisimple
element of $G'_n$ is then the same as for $G_n$. If $\pi\in \rrr'_n$, we let again
$\chi_\pi$ be the function character of $\pi$.
As for $G_n$, we will denote $\nu$ the character of $G'_n$ given by the absolute value of the determinant
(there will be no confusion with the one on $G_n$).

\subsection{Classification of $Irr_n$ (resp. $Irr_n'$) in terms of $\d_l$ (resp. $\d'_l$), $l\leq n$}

Let $\pi\in Irr_n$. There exist a \sgls\ $L=\times\ki G_{n_i}$ of $G_n$ and, for all $1\leq i\leq k$,
$\rho_i\in\cusp_{n_i}$,
such that $\pi$ is a subquotient of $\prod\ki\rho_i$. The non-ordered multiset of cuspidal representations
$\{\rho_1,\rho_2,...\rho_n\}$
is determined by $\pi$ and is called {\bf the cuspidal support of $\pi$}.

We recall the Langlands classification which takes a particularly nice form on $G_n$.
Let $L=\times\ki G_{n_i}$ be a \sgls\ of $G_n$ and $\s\in \d(L)=\times\ki \d_{n_i}$.
Let us write $\s=\otimes_{i=1}^k\s_i$ with
$\s_i\in \d_{n_i}$. For each $i$, write $\s_i=\nu^{e_i}\s_i^u$, where $e_i\in\r$ and $\s_i^u\in \d^u_{n_i}$.
Let $p$ be a permutation of the set $\{1,2,...,k\}$ such that the sequence $e_{p(i)}$ is decreasing.
Let $L_p=\times_{i=1}^k G_{n_{p(i)}}$ and $\s_p=\otimes_{i=1}^k\s_{p(i)}$. Then $ind_{L_p}^{G_n}\s_p$ has a
unique irreducible quotient $\pi$ and $\pi$ is independent of the choice of $p$ under the condition that
$(e_{p(i)})_{1\leq i\leq k}$ is decreasing. So $\pi$ is defined by the non ordered
multiset $\{\s_1,\s_2,...,\s_k\}$.
We write then $\pi=Lg(\s)$. Every $\pi\in Irr_n$ is obtained like this. If $\pi\in Irr_n$
and $L=\times\ki G_{n_i}$ and
$L'=\times_{j=1}^{k'}G_{n'_j}$ are
two \sglss\ of $G_n$, if $\s=\otimes\ki\s_i$, with $\s_i\in \d_{n_i}$, and $\s'=\otimes_{j=1}^{k'}\s'_j$, with
$\s'_j\in \d_{n'_j}$,
are such that $\pi=Lg(\s)=Lg(\s')$, then $k=k'$ and
there exists a permutation $p$ of  $\{1,2,...,k\}$ such that $n'_j=n_{p(i)}$ and $\s'_j=\s_{p(i)}$.
So the non ordered multiset $\{\s_1,\s_2,...,\s_k\}$ is determined by $\pi$ and it is called {\bf the
essentially square integrable support of $\pi$} which we abridge as  {\bf the esi-support of $\pi$}.

An element $S=\i_L^{G_n}\s$ of $\rrr_n$, with $\s\in \d(L)$, is called a {\bf standard representation} of $G_n$.
We will often write $Lg(S)$ for $Lg(\s).$
The set $\b_n$ of standard representations of $G_n$ is a basis of $\rrr_n$ and the map $S\mapsto Lg(S)$ is a bijection
from $\b_n$ onto $Irr_n$. All these results are consequences of the Langlands classification (see [Ze] and [Rod]).
We also have the following result: if for all $\pi\in Irr_n$ we write $\pi=Lg(S)$ for some standard
representation $S$ and then for all $\pi'\in Irr_n\bc\{\pi\}$ we set
$\pi'< \pi$ if and only if $\pi'$ is a subquotient of $S$, then we obtain a well defined partial order relation
on $Irr_n$.

The same definitions and theory, including the order relation, hold for $G'_n$ (see [Ta2]).
The set of standard representations of $G'_n$ is denoted here by
$\b'_n$.

For $G_n$ or $G'_n$ we have the following proposition , where $\s_1$ and $\s_2$ are essentially square integrable
representations:

\begin{prop}\label{reduc}
The representation $Lg(\s_1)\times Lg(\s_2)$ contains $Lg(\s_1\times\s_2)$ as a subquotient with multiplicity 1.
If $\pi$ is another irreducible subquotient of $Lg(\s_1)\times Lg(\s_2)$, then $\pi<Lg(\s_1\times\s_2)$.
In particular, if $Lg(\s_1)\times Lg(\s_2)$ is reducible, it has at least two different subquotients.
\end{prop}

See [Ze] and [Ta2] for the proof.

\subsection{Classification of $\d_n$ in terms of $\cusp_l$, $l|n$}\label{esi}

Let $k$ and $l$ be two positive integers and set $n=kl$.
Let $\rho\in \cusp_l$. Then the representation $\prod_{i=0}^{k-1}\nu^i\rho$ has a unique irreducible quotient $\s$.
$\s$ is an essentially square integrable representation
of $G_n$. We write then $\s=Z(\rho,k)$.
Every $\s\in\d_n$ is obtained like this and $l$, $k$ and $\rho$ are determined by $\s$.
This may be found in [Ze].

In general, a set $S=\{\rho, \nu\rho, \nu^2\rho,...,\nu^{a-1}\rho\}$, $\rho\in\cusp_b$, $a,b\in \n^*$,
is called a {\bf segment},
$a$ is {\bf the length} of the segment $S$ and $\nu^{a-1}\rho$ is the {\bf ending} of $S$.

\subsection{Local Jacquet-Langlands correspondence}

Let $n\in \n^*$. Let $g\in G_{nd}$ and $g'\in G'_n$. We say that $g$ {\bf corresponds}
to $g'$ if $g$ and $g'$ are
regular semisimple and have the same characteristic polynomial. We shortly write then $g\lra g'$.

\begin{theo}
There is a unique bijection $\ccc:\d_{nd}\to \d'_n$ such that for all $\pi\in\d_{nd}$ we have
$$\chi_\pi(g)=(-1)^{nd-n}\chi_{\ccc(\pi)}(g')$$
for all $g\in G_{nd}$ and $g'\in G'_n$ such that $g\lra g'$.
\end{theo}

For the proof, [DKV] if the characteristic of the base field $F$ is zero and [Ba2] for the non zero characteristic case.
I should quote here too the particular cases [JL], [Fl2] and [Ro] which contain some germs of the general
proof in [DKV].\\

We identify the centers of $G_{nd}$ and $G'_n$ via the canonical isomorphism. Then the
correspondence $\ccc$ preserves central characters so in particular $\s\in \d^u_{nd}$ if and only
if $\ccc(\s)\in \d^u_{n}$.

If $L'=\times\ki G'_{n_i}$ is a \sgls\ of $G'_n$ we say that the \sgls\ $L=\times\ki G_{dn_i}$ of $G_{nd}$
{\bf corresponds} to $L'$.
Then the Jacquet-Langlands correspondence extends in an obvious way to a bijective correspondence
$\d (L)$ to $\d'(L')$ with the same properties. We will denote this correspondence by the same letter $\ccc$.
A \sgls\ $L$ of
$G_n$ corresponds to a \sgls\ or $G'_r$ if and only if it is defined by a sequence $(n_1,n_2,...,n_k)$ such that
each $n_i$ is divisible by $d$. We then say $L$ {\bf transfers}.

\subsection{Classification of $\d'_n$ in terms of $\cusp'_l$, $l| n$. The invariant $s(\s')$}\label{esi2}

Let $l$ be a positive integer and $\rho'\in \cusp'_l$. Then $\s=\ccc^{-1}(\rho')$ is  an essentially
square integrable
representation of $G_{ld}$. We may write
$\s=Z(\rho,p)$ for some $p\in\n^*$ and some $\rho\in\cusp_{\frac{ld}{p}}$.
Set then $s(\rho')=p$ and $\nu_{\rho'}=\nu^{s(\rho')}$.

Let $k$ and $l$ be two positive integers and set $n=kl$.
Let $\rho'\in \cusp'_l$. Then the representation
$\prod_{i=0}^{k-1}\nu_{\rho'}^i\rho'$ has a unique irreducible quotient
$\s'$. $\s'$ is an essentially square integrable representation
of $G'_n$. We write then $\s'=T(\rho',k)$.
Every $\s'\in\d'_n$ is obtained like this and $l$, $k$ and $\rho'$ are determined by $\s'$. We set then
$s(\s')=s(\rho')$. For this classification see [Ta2].

A set $S'=\{\rho', \nu_{\rho'}\rho', \nu_{\rho'}^2\rho',...,\nu_{\rho'}^{a-1}\rho'\}$,
$\rho'\in\cusp'_b$, $a,b\in \n^*$,
is called a {\bf segment}, $a$ is {\bf the length} of $S'$ and $\nu_{\rho'}^{a-1}\rho'$ is the {\bf ending} of $S'$.

\subsection{Multisegments, order relation, the function $\l$ and rigid representations}\label{multisegments}

Here we will give the definitions and results in terms of groups $G_n$, but one may replace $G_n$ by
$G'_n$.
We have seen (section  \ref{esi} and \ref{esi2}) that
 to each $\s\in\d_n$ one may associate a segment.
A multiset of segments is called a {\bf multisegment}. If $M$ is a multisegment, the multiset of endings of its
elements (see section  \ref{esi} and \ref{esi2} for the definition) is denoted $E(M)$.

If $\pi\in G_n$,
the multiset of the segments of the elements of the
esi-support of $\pi$ is a multisegment; we will denote it by $M_\pi$. $M_\pi$ determines $\pi$.
The reunion with repetitions of the elements of $M_\pi$ is the cuspidal support of
$\pi$.

 Two segments $S_1$ and $S_2$
are said to be {\bf linked} if $S_1\cup S_2$ is a segment different from $S_1$ and $S_2$.
If $S_1$ and $S_2$ are linked, we say they are {\bf adjacent} if $S_1\cap S_2=\O$.

Let $M$ be a multisegment, and assume $S_1$ and
$S_2$ are two linked segments in $M$. Let $M'$ be the multisegment defined by

- $M'=(M\cup \{S_1\cup S_2\}\cup \{S_1\cap S_2\}) \bc\{S_1,S_2\}$ if $S_1$ and $S_2$ are not adjacent (i.e.
$S_1\cap S_2\neq \O$), and

- $M'=(M\cup \{S_1\cup S_2\})\bc\{S_1,S_2\}$ if $S_1$ and $S_2$ are adjacent (i.e. $S_1\cap S_2= \O$).

We say that we made an {\bf elementary operation} on $M$ to get $M'$,
or that $M'$ was
obtained from $M$ by an elementary operation. We then say $M'$ is inferior to $M$. It is easy to verify
this extends by transitivity to a well defined partial order relation $<$ on the set of multisegments of $G_n$.
The following proposition  is a result of [Ze] for $G_n$ and [Ta2] for $G'_n$.

\begin{prop}
If $\pi,\pi'\in Irr_n$, then $\pi<\pi'$ if and only if $M_\pi<M_{\pi'}$.
\end{prop}

If $\pi<\pi'$, then the cuspidal support of $\pi$
equals the cuspidal support of $\pi'$.

Define a function $\l$ on the set of multisegments by: if $M$ is a multisegment, then $\l(M)$ is the maximum
of the lengths of the segments in $M$. If $\pi\in Irr_n$, set $\l(\pi)=\l(M_\pi)$.
The following lemma is obvious:

\begin{lemme}\label{length}
If $M'$ is obtained from $M$ by an elementary operation then $\l(M)\leq \l(M')$ and $E(M')\subset E(M)$.
As a function on $Irr_n$, $\l$ is decreasing.
\end{lemme}

The next important proposition  is also a result from [Ze] and [Ta2]:

\begin{prop}\label{irred}
Let $\pi\in Irr_k$ and $\pi'\in Irr_l$.
If for all $S\in M_\pi$ and $S'\in M_{\pi'}$ the segments $S$ and $S'$ are not linked, then
$\pi\times\pi'$ is irreducible.
\end{prop}

There is an interesting consequence of this last proposition . Let $l\in \n^*$ and $\rho\in \cusp_l$. We will call
the set $X=\{\nu^a\rho\}_{a\in\z}$ a {\bf line}, the line generated by $\rho$. Of course $X$ is also the line
generated by
$\nu\rho$ for example. If $\pi\in Irr_n$, we say $\pi$ is {\bf rigid} if the set of elements of the cuspidal support of
$\pi$ is included in one single line. As a consequence of the previous proposition  we have the

\begin{cor}\label{rigid}
Let $\pi\in Irr_n$. Let $X$ be the set of the elements of the cuspidal support of $\pi$.
If $\{D_1,D_2,...,D_m\}$ is the set
of all the lines with which $X$ has a non empty intersection, then one may write in a unique (up to permutation) way
$\pi=\pi_1\times\pi_2\times...\times\pi_m$ with $\pi_i$ rigid irreducible and the set of elements of the cuspidal
support of $\pi_i$ included in $D_i$, $1\leq i\leq m$.
\end{cor}

We will say $\pi=\pi_1\times\pi_2\times...\times\pi_m$ is the {\bf standard decomposition} of $\pi$ in a
product of rigid representations (this is only {\it the shortest} decomposition of
$\pi$ as a product of rigid representations, but there might exist finer ones).

The same hold for $G'_n$.

\subsection{The involution}

Aubert defined in [Au] an involution (studied too by
Schneider and Stuhler in [ScS])
of the group of Grothendieck of smooth representations of finite length
of a reductive group over a local non-Archimedean field. The involution
sends an irreducible representation to an irreducible representation up to a sign.
We specialize this involution to $G_n$, resp. $G'_n$, and denote it $i_n$, resp. $i'_n$. We will
write $i$ and $i'$ when the index is not relevant or it is clearly understood.
With this notation we have the relation $i(\pi_1)\times i(\pi_2)=i (\pi_1\times\pi_2)$, i.e. the
``the involution commutes with the parabolic induction''. The same holds for $i'$. The reader may find  all this facts
 in [Au].

If $\pi\in Irr_n$, then one and only one among $i(\pi)$ and
$-i(\pi)$ is an irreducible representation. We denote it by
$|i(\pi)|$. We denote $|i|$ the involution of $Irr_n$ defined by
$\pi\mapsto |i(\pi)|$. The same facts and definitions hold for $i'$.

In [MW1] is proven the algorithm conjectured by Zelevinsky for to compute the esi-support of $|i(\pi)|$ from
the esi-support of $\pi$ when $\pi$ is rigid (and hence more generally for
$\pi\in Irr_n$, cf. corollary \ref{rigid}). The same facts and algorithm hold for $|i'|$ as explained in [BR2].

\subsection{The extended correspondence}\label{extended}

The correspondence $\ccc^{-1}$ may be extended in a natural way to a correspondence $\lj$ between Grothendieck groups.
Let $S'=\i_{L'}^{G'_n}\s'\in \b'_n$, where $L'$ is a \sgls\ of $G'_n$ and
$\s'$ an essentially square integrable representation of $L$. Set $M_n(S')=\i_{L}^{G_{nd}}\ccc^{-1}(\s')$, where $L$
is the \sgls\ of $G_{nd}$ corresponding to $L'$. Then $M_n(S')$ is a standard representation of $G_{nd}$ and
$M_n$ realizes an injective map from $\b'_n$ into $\b_{nd}$. Define $Q_n:Irr'_n\to Irr_{nd}$ by $Q_n(Lg(S'))=Lg(M_n(S'))$.
If $\pi'_1<\pi'_2$, then $Q_n(\pi'_1)<Q_n(\pi'_2)$. So $Q_n$ induces on $Irr(G'_n)$, by transfer from $G_{nd}$,
an order relation $<<$ which is stronger than $<$.

Let $\lj_{n}:\rrr_{nd}\to\rrr'_n$ be the $\z$-morphism defined on $\b_{nd}$ by setting
$\lj_n(M_n(S'))=S'$ and $\lj_n(S)=0$ if $S$ is
not in the image of $M_n$.

\begin{theo}\label{eu}
(a) For all $n\in \n^*$, $\lj_{n}$ is the unique map from $\rrr_{nd}$ to $\rrr'_n$
such that for all $\pi\in\rrr_{nd}$ we have
$$\chi_\pi(g)=(-1)^{nd-n}\chi_{\lj_n(\pi)}(g')$$
for all $g\lra g'$.

(b) The map $\lj_n$ is a surjective group morphism.

c) One has
$$\lj_n(Q_n(\pi'))=\pi'+\sum_{\pi'_j<<\pi'}b_j\pi'_j$$
where $b_j\in \z$ and $\pi'_j\in Irr(G'_n)$.

(d) One has
$$\lj_{n}\circ i_{nd}=(-1)^{nd-n}i'_n\circ\lj_n.$$
\end{theo}

See [Ba4]. We will often drop the index and write only $Q$, $M$ and
$\lj$. $\lj$ may be extended in an obvious way to \sglss. For a
\sgls\ $L'$ of $G'_n$ which correspond to a \sgls\ $L$ of $G_{nd}$
we have $\lj\circ\i_L^{G_{nd}}= \i_{L'}^{G'_n}\circ\lj$.

We will say that $\pi\in \rrr_{nd}$ is $d$-{\bf compatible} if
$\lj_n(\pi)\neq 0$. This means that $\chi_\pi$ is not zero on all
regular semisimple elements of $G_{nd}$ which correspond to an
element of $G'_{n}$. A regular semisimple element of $G_{nd}$
correspond to an element of $G'_{n}$ if and only if its
characteristic polynomial decomposes in irreducible factors all of
which the degrees are divisible by $d$. So our definition
 depends only on $d$, not on $D$. A product of representations is
 $d$-compatible if and only if each factor is
 $d$-compatible.

\subsection{Unitary representations of $G_n$}

We are going to use the word {\bf unitary} for {\bf unitarizable}.
Let $k$, $l$ be positive integers and set $kl=n$.

Let $\rho\in \cusp_l$ and
set $\s=Z(\rho,k)$. Then $\s$ is unitary if and only if $\nu^\frac{k-1}{2}\rho$ is unitary.
We set then
$\rho^u=\nu^\frac{k-1}{2}\rho\in \cusp^u_l$ and we write $\s=Z^u(\rho^u,k)$. From now on, anytime we write
$\s=Z^u(\rho,k)$, it is understood that $\s$ and $\rho$ are unitary.

Now, if  $\s\in \d^u_l$, we set
$$u(\sigma,k)=Lg(\prod_{i=0}^{k-1} \nu^{\frac{k-1}{2}-i}\s).$$
The representation $u(\s,k)$ is an irreducible representation of $G_n$.

If $\a\in ]0,\frac{1}{2}[$, we moreover set
$$\pi(u(\s,k),\a)=\nu^{\a}\s\times\nu^{-\a}\s.$$
The representation $\pi(u(\s,k),\a)$
is an irreducible representation of $G_{2n}$ (by proposition  1.4).

Let us recall the Tadi\'c
classification of unitary representations in
[Ta1].

Let $\mathcal U$ be the set of all the representations $u(\s,k)$ and $\pi(u(\s,k),\a)$ where $k,l$ range over $\n^*$,
$\s\in \cusp_l$ and $\a\in ]0,\frac{1}{2}[$. Then any product of elements of $\mathcal U$ is irreducible and unitary.
 Every irreducible unitary representation  $\pi$ of some $G_n$, $n\in \n^*$, is such a product.
The non ordered multiset of the factors of the product are determined by $\pi$.

The fact that a product of irreducible unitary representations is irreducible is due to Bernstein ([Be]).

Tadi\'c computed the decomposition of the representation $u(\s,k)$ on the basis $\b_n$ of
$\rrr_n$.

\begin{prop}\label{formula1}{\rm ([Ta4])}
Let $\s=Z(\rho,l)$ and $k\in\n^*$.
Let $W_k^l$ be the set of permutations $w$ of $\{1,2,...,k\}$ such that $w(i)+l\geq i$ for all $i\in\{1,2,...,k\}$.
Then we have:
$$
u(\s,k)=\nu^{-\frac{k+l}{2}}(\sum_{w\in W_k^l}(-1)^{sgn(w)}\prod_{i=1}^kZ(\nu^{i}\rho,w(i)+l-i)).$$
\end{prop}

One can also compute  the dual of $u(\s,k)$.

\begin{prop}\label{dual}
Let $\s=Z^u(\rho^u,l)$ and $k\in \n^*$.
If $\tau=Z^u(\rho^u,k)$, then
$$|i(u(\s,k))|=u(\tau,l).$$
\end{prop}

This is the theorem 7.1 iii) [Ta1], and also a consequence of [MW1].

\subsection{Unitary representations of $G'_n$}

Let $k,l\in \n^*$ and set $n=kl$. Let $\rho\in\cusp'_l$ and  $\s'=T(\rho',k)\in \d'_n$. As for $G_n$, one has
$\s'\in\d'^u_n$  if and only if
$\nu_{\rho'}^\frac{k-1}{2}\rho'$ is unitary; we set then
$\rho'^{u}=\nu^\frac{k-1}{2}\rho'$ and write $\s'=T^u(\rho'^{u},k)$.

If now $\s'\in \d^{'u}_l$, we set

$$u'(\s',k)=
Lg(\prod_{i=0}^{k-1} \nu_{\s'}^{i-\frac{k-1}{2}}\s')$$

and

$$\bar{u}(\s',k)=
Lg(\prod_{i=0}^{k-1} \nu^{i-\frac{k-1}{2}}\s').$$
The representations $u'(\s',k)$ and $\bar{u}(\s',k)$ are irreducible representations of $G'_n$.

If moreover $\a\in ]0,\frac{1}{2}[$, we  set
$$\pi(u'(\s',k),\a)=\nu_{\s'}^{\a}\s'\times\nu_{\s'}^{-\a}\s'.$$
The representation $\pi(u'(\s',k),\a)$
is an irreducible representation of $G'_{2n}$ (cf. [Ta2]; consequence of the (restated) proposition  \ref{irred} here).

We have the formulas:

\begin{equation}\label{equa2}
\bar{u}(\s',ks(\s'))=(\prod_{i=1}^{s(\s')}\nu^{i-\frac{s(\s')+1}{2}}u'(\s',k));
\end{equation}

and, for all integer $1\leq b\leq s(\s')-1$,

\begin{equation}\label{equa1}
\bar{u}(\s',ks(\s')+b)=(\prod_{i=1}^{b}\nu^{i-\frac{b+1}{2}}u'(\s',k+1))
\times
(\prod_{j=1}^{s(\s')-b}\nu^{j-\frac{s(\s')-b+1}{2}}u'(\s',k)),
\end{equation}
with the convention that we make abstraction of the second product if $k=0$.

The  products are irreducible because  the segments appearing in the
esi-support of two different factors are never linked and the proposition  \ref{irred}.
The fact that the product is indeed $\bar{u}(\s',ks(\s'))$ (and resp. $\bar{u}(\s',ks(\s')+b)$) is
then clear by proposition  \ref{reduc}. This kind of formulas has been used (at least) in [BR1] and [Ta6].

The representations $u'(\s',k)$ and $\bar{u}(\s',k)$ are known to be
unitary at least in zero characteristic ([Ba4] and [BR1]).

One has

\begin{prop}\label{br}
Let $\s'=Z^u(\rho'^u,l)$ and $k\in \n^*$.
If $\tau'=Z^u(\rho'^u,k)$, then

(a) $|i'(u'(\s,k))|=u'(\tau,l)$
and

(b) $|i'(\bar{u}(\s,ks(\s')))|=\bar{u}(\tau,ls(\s')).$
\end{prop}

{\bf Proof.} The point (a) is a direct consequence of [BR2]. For the point (b), it is enough
to use the relation \ref{equa2}, the point (a) here and the fact that $i'$ commutes with parabolic induction.\qed

\subsection{Hermitian representations and an irreducibility trick}\label{hermrep}

If $\pi\in Irr'_n$, write $h(\pi)$ for the complex conjugated representation of the
 contragredient of ${\pi}$. A representation $\pi\in Irr'_n$ is called {\bf hermitian} if  $\pi=h(\pi)$
(we recall, to avoid confusion, that here we use ``$=$'' for the usual ``equivalent''). A unitary representation
is always hermitian.
If $A=\{\s_i\}_{1\leq i\leq k}$ is a multiset  of essentially square integrable representations of some $G'_{l_i}$,
we define the multiset $h(A)$ by  $h(A)=\{h(\s_i)\}_{1\leq i\leq k}$.
If $\pi\in Irr_n$ and $x\in\r$, then $h(\nu^x\pi)=\nu^{-x}h(\pi)$, so
if $\s'\in\d'_l$ and we write $\s'=\nu^{e}\s'^{u}$
with $e\in \r$ and $\s'^{u}\in\d'^{u}_l$, then $h(\s')=\nu^{-e}\s'^{u}\in \d'_l$.
An easy consequence of proposition  3.1.1 in [Ca] is the

\begin{prop}\label{casselman}
 If $\pi\in Irr'_n$, and $A$ is the esi-support of $\pi$, then $h(A)$ is the esi-support of $h(\pi)$.
In particular,
$\pi$ is hermitian if and only if the esi-support $A$ of $\pi$ verifies $h(A)=A$.
\end{prop}

Let us give a lemma.

\begin{lemme}\label{hermitian}
Let $\pi_1\in Irr'_{n_1}$ and $\pi_2\in Irr'_{n_2}$ and assume $h(\pi_1)\neq \pi_2$. Then
there exists $\e>0$ such that for all $x\in]0,\e[$
the representation $a_x=\nu^x\pi_1\times \nu^{-x}\pi_2$ is irreducible, not hermitian.
\end{lemme}

{\bf Proof.}  For all $x\in\r$ let $A_x$ be the esi-support of $\nu^x\pi_1$ and $B_x$ be
the esi-support of $\nu^{-x}\pi_2$.
Then the set $X$ of $x\in\r$ such that  $A_x\cap h(A_x)\neq\emptyset$
or $B_x\cap h(B_x)\neq\emptyset$ is finite
(it is enough to check the central character of the representations in these multisets).
The set $Y$ of $x\in\r$ such that  the cuspidal supports of
$A_x$ and $B_x$ has non empty intersection is finite too. Now, if $x\in \r\bc Y$,
$a_x$ is irreducible by the proposition  \ref{irred}. Assume moreover $x\notin X$.
As $a_x$ is irreducible, if it were hermitian one should have $h(A_x)\cup h(B_x)=A_x\cup B_x$
(where the reunions are to be taken
with multiplicities, as reunions of multisets) by the proposition  \ref{casselman}. But if $A_x\cap h(A_x)=\emptyset$
and $B_x\cap h(B_x)=\emptyset$, then this would lead to $h(A_x)=B_x$, and hence to $h(\pi_1)=\pi_2$
which contradicts the hypothesis.\qed

We now state our irreducibility trick.

\begin{prop}\label{trick}
Let $u'_i\in Irr'^u_{n_i}$, $i\in\{1,2,...,k\}$. If, for all $i\in\k$, $u'_i\times u'_i$ is irreducible, then
$\prod_{i=1}^ku'_i$ is
irreducible.
\end{prop}

{\bf Proof.} There exists $\e>0$ such that for all $i\in\k$
the cuspidal supports of \,  $\nu^xu'_i$ \, and of \, $\nu^{-x}u'_i$ \, are disjoint for all
$x\in ]0,\e[$.
Then, for all $i\in\k$, for all $x\in ]0,\e[$, the representation \, $\nu^xu'_i\times \nu^{-x}u'_i$ \,
is irreducible. As, by hypothesis, $u'_i\times u'_i$ is irreducible and unitary, that implies that
for all $x\in ]0,\e[$,
the representation \, $\nu^xu'_i\times \nu^{-x}u'_i$ \,
is also unitary (see for example [Ta3], section  (b)).
So \, $\prod\ki\nu^xu'_i\times \nu^{-x}u'_i$ \, is a sum of unitary representations.
But we have (in the Grothendieck group)
$$\prod_{i=1}^k(\nu^xu'_i\times \nu^{-x}u'_i) = (\nu^x\prod_{i=1}^k u'_i)\times (\nu^{-x}\prod_{i=1}^k u'_i).$$

Let us assume by absurd \, $\prod\ki u'_i$ \, is reducible.
Then it contains at least two \underline{different} unitary subrepresentations $\pi_1$ and $\pi_2$
(proposition  \ref{reduc}).
Then \, $(\nu^x\prod\ki u'_i)\times (\nu^{-x}\prod\ki u'_i)$ \, contains \, $\nu^x\pi_1\times \nu^{-x}\pi_2$ \,
as a subquotient for some $x\in ]0,\e[$ for which this representation is irreducible not hermitian
(by lemma \ref{hermitian}). We found an irreducible non unitary subquotient which contradicts our assumption.
\qed

\section{Local results}

\subsection{First results}

Let $\s'\in \d'^u_n$ and set $\s=\ccc^{-1}(\s')\in \d^u_{nd}$.
Write $\s'=T^u(\rho',l)$ for some $l\in \n^*$ and $\rho'\in \cusp_l$.

Let $k$ be a positive integer and set $k'=ks(\s')$.
 Then we have the
following:

\begin{theo}\label{main}

(a) One has
$$\lj(u(\s,k'))=\bar{u}(\s',k').$$

(b) The induced representation $\bar{u}(\s',k')\times \bar{u}(\s',k')$ is irreducible.

(c) Let $W_k^l$ be the set of permutation $w$ of $\{1,2,...,k\}$ such that $w(i)+l\geq i$ for all $i\in\{1,2,...,k\}$.
Then we have

$$\bar{u}(\s',k')=\nu^{-\frac{k'+l'}{2}-\frac{s(\s')-1}{2}}(\sum_{w\in W_k^l}(-1)^{sgn(w)}
\prod_{i=1}^{k'} T(\nu^i\rho',w(i)+l-i)).$$
\end{theo}

{\bf Proof.}
(a) Set $\tau'=T^u(\rho',k)$ and $l'=ls(\s')$.
Then we have $\s=Z^u(\rho,l')$ and
$\tau=\ccc^{-1}(\tau')=Z^u(\rho,k')$ for some unitary cuspidal representation $\rho$ defined by
$\ccc(Z^u(\rho,s(\s')))=\rho'$.

We apply  the theorem \ref{eu} (c) to $\bar{u}(\s',k')$ and
$\bar{u}(\tau',l')$. We get

\begin{equation}\label{alfa1}
\lj(u(\s,k'))=\bar{u}(\s',k')+\sum_{\pi'_j<<\bar{u}(\s',k')}b_j\pi'_j
\end{equation}

and

\begin{equation}\label{alfa2}
\lj(u(\tau,l'))=\bar{u}(\tau',l')+\sum_{\tau'_q<<\bar{u}(\tau',l')}c_q\tau'_q
\end{equation}

We want to show that all the $b_j$ vanish.

Let us write the dual equation to \ref{alfa1} (cf. theo. \ref{eu} (d)).
 As
$|i(u(\s,k'))|=u(\tau,l')$ (proposition  \ref{dual})
and $|i'(\bar{u}(\s',k'))|=\bar{u}(\tau',l')$ (proposition  \ref{br}), we obtain:

\begin{equation}\label{alfa3}
\lj(u(\tau,l'))=\varepsilon_1 \bar{u}(\tau',l')  +\varepsilon_2 \sum_{\pi'_j<<\bar{u}(\s',k')}b_ji'(\pi'_j).
\end{equation}

for some signs $\varepsilon_1,\varepsilon_2\in\{-1,1\}$.
The equations \ref{alfa2} and \ref{alfa3} imply then the equality:
\begin{equation}\label{alfa4}
\bar{u}(\tau',l')+\sum_{\tau'_q<<\bar{u}(\tau',l')}c_{q}\tau'_q
=\varepsilon_1 \bar{u}(\tau',l')  +\varepsilon_2\ (\sum_{\pi'_j << \bar{u}(\s',k')}b_ji'(\pi'_j)).
\end{equation}
First, remark that since $\pi'_j\neq \bar{u}(\s',k')$ for all $j$,
we also have $|i'(\pi'_j)|\neq \bar{u}(\tau',l')$ for all $j$. So by
linear independence of irreducible representations in the
Grothendieck group, $\varepsilon_1=1$ and the term
$\bar{u}(\tau',l')$ cancel.

We will now show that the remaining equality

$$\sum_{\tau'_q<<\bar{u}(\tau',l')}c_{q}\tau'_q=\varepsilon_2\ (\sum_{\pi'_j << \bar{u}(\s',k')}b_ji'(\pi'_j)).$$
implies that all the coefficients $b_j$ vanish. The argument is the
linear independence of irreducible representations and the lemma:

\begin{lemme}\label{impossible}
If $\pi'_j<<\bar{u}(\s',k')$, it
is impossible to have $|i'(\pi'_j)|<<\bar{u}(\tau',l')$.
\end{lemme}

{\bf Proof.} The proof is complicated by the fact that we do not have in general equality $<\ =\ <<$ between the
order relations. But
this does not really matter.
Recall that $\pi'_j<<\bar{u}(\s',k')$, means by definition $Q(\pi_j)<Q(\bar{u}(\s',k'))$,
i.e. there exists $\pi_j<u(\s,k')$
such that the esi-support of $\pi'_j$
corresponds to the esi-support of $\pi_j$ element by element
by Jacquet-Langlands. This implies the only two properties we need:

(*) the cuspidal support of $\pi'_j$
equals the cuspidal support of $\bar{u}(\s',k')$ and

(**) we have the inclusion relation $E(M_{\pi'_j})\subset E(M_{\bar{u}(\s',k')})$ (lemma \ref{length}).

The property (*) imply in that, if
$$\pi'_j=a_1\times a_2\times...\times a_{x}$$
is a standard decomposition of $\pi'_j$ in a product of rigid
representations, then:

- $x=s(\s')$,

- we may assume that for $1\leq t\leq s(\s')$ the line of $a_t$ is
generated by $\nu^t\rho'$ and

- the multisegment $M_t$ of $a_t$ has maximum $k$ elements.\\
So, if
one uses the Zelevinsky-Moeglin-Waldspurger algorithm to compute the
esi-support $M_t^\#$ of $|i'(a_t)|$ (cf. [BR2]), one finds that $\l
(M_t^\#)\leq k$, since each segment in $M_t^\#$ is constructed by
picking up at most one cuspidal representation from each segment in
$M_t$. This implies that $\l(|i'(a_t)|)\leq k$. As
$$|i'(\pi'_j)|=|i'(a_1)|\times |i'(a_2)|\times...\times |i'(a_{x})|$$
we eventually have $\l(|i'(\pi'_j)|)\leq k$.

Assume now $|i'(\pi'_j)|<<\bar{u}(\tau',l')$. We will show that $\l(|i'(\pi'_j)|)> k$.
Set $Q(|i'(\pi'_j)|)=\gamma$ and we know that $\gamma<u(\tau,l')$. We obviously have in our particular situation
$\l(\gamma)=s(\s')\l(|i'(\pi'_j)|)$.
So we want to prove $\l(\gamma)>k'$.
The multisegment of $\gamma$ is obtained by a sequence of elementary operation from the
multisegment of
$u(\tau,l')$: at the first elementary operation on the multisegment of $u(\tau,l')$ we get a
multisegment $M'$ such that $\l(M')>k'$ and then we apply the lemma \ref{length}.
We get, indeed, $\l(\gamma)>k'$.

So our assumption leads to a
contradiction.\qed
\ \\

(b) The proof is similar to the one at the point (a), but uses it.
Let $\tau$ and $\tau'$ be defined like in (a).
By the point (a) we know now that
$$\lj(u(\s,k'))=\bar{u}(\s',k')\ \ \ \ \ \ \ \ \text{and}\ \ \ \ \ \ \ \ \lj(u(\tau,l'))=\bar{u}(\tau',l'),$$
so
$$\lj(u(\s,k')\times u(\s,k'))=\bar{u}(\s',k')\times \bar{u}(\s',k')$$
and
$$\lj(u(\tau,l')\times u(\tau,l'))=\bar{u}(\tau',l')\times \bar{u}(\tau',l').$$
Let us call $K_1$ the Langlands quotient of the esi-support of $\bar{u}(\s',k')\times \bar{u}(\s',k')$
and $K_2$ the Langlands quotient of the esi-support of $\bar{u}(\tau',l')\times \bar{u}(\tau',l')$. Using [BR2] it
is easy to see that $|i'(K_1)|=K_2$.
Then we may write, using theo. \ref{eu} (c) and proposition  \ref{reduc}:
\begin{equation}\label{prima}
\lj(u(\s,k')\times u(\s,k'))=K_1+\sum_{\pi_j<<K_1}b_j\pi'_j
\end{equation}
and
\begin{equation}\label{adoua}
\lj(u(\tau,l')\times u(\tau,l'))=K_2+\sum_{\xi'_m<<K_2}c_m\xi'_m.
\end{equation}
We want to prove that all the $b_j$ vanish. Let us take the dual in the equation \ref{prima} (cf. proposition  \ref{eu} (d)):

\begin{equation}\label{atreia}
\lj(i(u(\s,k')\times u(\s,k')))=\pm(i'(K_1)+\sum_{\pi_j<<K_1}b_ji'(\pi'_j)).
\end{equation}
We know that $|i(u(\s,k')\times u(\s,k'))|=u(\tau,l')\times u(\tau,l'))$ because $i$ commutes to the induction
functor and we have
$|i(u(\s,k'))|=u(\tau,l')$ by proposition  \ref{dual}.
As  $|i'(K_1)|=K_2$, we get from equations \ref{adoua} and \ref{atreia} after simplification with $K_2$
(as in the equation
\ref{alfa4}):

$$\sum_{\pi_j<<K_1}b_ji'(\pi'_j)=\pm(\sum_{\xi'_m<<K_2}c_m\xi'_m).$$
To show that all the $b_j$ vanish, it is enough by linear
independence of irreducible representations to show the following:

\begin{lemme}
If $\pi'<K_1$ it is impossible to have $|i'(\pi')|<K_2$.
\end{lemme}

{\bf Proof.} The proof of the lemma \ref{impossible} apply here with a minor modification.
We write again
$$\pi'=a_1\times a_2\times...\times a_{s(\s')}$$
such
 that the line of $a_t$, $1\leq t\leq s(\s')$,
is generated by $\nu^t\rho'$. The difference here is that the multisegment $M$
of $a_t$ may have up to $2k$ elements.
We will prove though, in this case again:

\begin{lemme}
The multisegment $m^\#$ of $|i'(a_t)|$  verifies $\l (m^\#)\leq k$.
\end{lemme}

This implies that $\l(\pi')\leq k$ and the rest of the proof goes the same as for (a).\\

{\bf Proof.}
Let us denote $m$ the multisegment of $a_t$ ($m$ and $m^\#$ respect the notations in [MW1]).
The multisegment $m^\#$ is obtained from $m$ using the
algorithm in [MW1] (cf. [BR2]). As $\pi'<<K_1$, one has $E(m)\subset \{\nu_{\rho'}^{\frac{l-k}{2}+1}\rho',
\nu_{\rho'}^{\frac{l-k}{2}+2}\rho',...,\nu_{\rho'}^\frac{l+k}{2}\rho'\}$ (it is the property (**)
given at the beginning of
the proof of the lemma \ref{impossible}). One
constructs all the segments of $m^\#$ ending with $\nu_{\rho'}^\frac{l+k}{2}\rho'$ using only cuspidal representations in
$E(m)$ (remark II.2.2 in  [MW1]). So the length of the constructed segments is at most $k$.
Let $m^-$ be the multisegment obtained from $m$ after we dropped off each segment of $m$ the cuspidal representations
used in this construction. We obviously have then $E(m^-)\subset \{\nu_{\rho'}^{\frac{l-k}{2}}\rho',
\nu_{\rho'}^{\frac{l-k}{2}+2}\rho',...,\nu_{\rho'}^{\frac{l+k}{2}-1}\rho'\}$ which has {\it again} $k$ elements.
So going through the algorithm
we will find that all the segments of $m^\#$ have length at most $k$.\qed
\\

(c) The point  (a) we have just proven allows us to transfer the
formula of the proposition  \ref{formula1} by $\lj$.

We have
$$\lj(u(\s,k'))=\nu^{-\frac{k'+l'}{2}}
(\sum_{w\in W_{k'}^{l'}}(-1)^{sgn(w)}\lj(\prod_{i=1}^{k'}Z(\nu^{i}\rho,w(i)+l'-i))).$$

The representations $\prod_{i=1}^{k'}Z(\nu^{i}\rho,w(i)+l'-i)$ are standard.
If
$w$ is such that, for some $i$,  $s(\s')$ does not divide
$w(i)-i$, then  $\lj(\prod_{i=1}^{k'}Z(\nu^{i}\rho,w(i)+l'-i))=0$.

If $w$ satisfies $s(\s')|w(i)-i$ for all $i$,
then $$\lj(\prod_{i=1}^{k'}Z(\nu^{i}\rho,w(i)+l'-i))=
\prod_{i=1}^{k'}T(\nu^{i-\frac{s(\s')-1}{2}}\rho',\frac{w(i)-i}{s(\s')}+l)).$$

In order to get the claimed formula one has roughly speaking to remark that, if $w$ satisfies
$s(\s')|w(i)-i$ for all $i$, then $w$ permutes the set of numbers between $1$ and $k$ which are
equal to a given number modulo $s(\s')$, that $w$ is determined by these induced permutations and that its signature
is the product of their signatures. This is lemma 3.1 in [Ta5].\qed



\begin{cor}\label{next}
Let $n,k\in \n^*$ and $\s'=\d'^u_n$.

(a) $u'(\s',k)\times u'(\s',k)$ is irreducible. $\pi(u'(\s',k),\a)$ are unitary.

(b) Write $\s'=T^u(\rho',l)$ for some unitary cuspidal representation $\rho'$.
Let $W_k^l$ be the set of permutation $w$ of $\{1,2,...,k\}$ such that $w(i)+l\geq i$ for all $i\in\{1,2,...,k\}$.
Then we have:
$$u'(\s',k)=\nu_{\s'}^{-\frac{k+l}{2}}(\sum_{w\in W_k^l}(-1)^{sgn(w)}
\prod_{i=1}^kT(\nu_{\s'}^{i}\rho',w(i)+l-i))$$
\end{cor}

{\bf Proof.} (a) That $u'(\s',k)\times u'(\s',k)$ is irreducible is clear from the point (b) of the theorem
\ref{main} and the formula \ref{equa2}.
The fact that this imply that all the $\pi(u'(\s',k),\a)$ are unitary is explained in [Ta2].

b) We want to show that
$$u'(\s',k)=\nu_{\s'}^{-\frac{k+l}{2}}(\sum_{w\in W_k^l}(-1)^{sgn(w)}
\prod_{i=1}^kT(\nu_{\s'}^{i}\rho',w(i)+l-i)).$$
We exploit the equality
$$\bar{u}(\s',ks(\s'))=\prod_{j=1}^{s(\s')}\nu^{j-\frac{s(\s')+1}{2}}u'(\s',k)$$
and the character formula for $\bar{u}(\s',ks(\s'))$ obtained at theorem \ref{main} (c).

Set
$$U=\nu_{\s'}^{-\frac{k+l}{2}}(\sum_{w\in W_k^l}(-1)^{sgn(w)}
\prod_{i=1}^kT(\nu_{\s'}^{i}\rho',w(i)+l-i))\in \rrr'_n.$$

We have

$$\prod_{j=1}^{s(\s')}\nu^{j-\frac{s(\s')+1}{2}}U=$$
$$
=\nu^{-\frac{k+l}{2}s(\s')}
\prod_{j=1}^{s(\s')}\nu^{j-\frac{s(\s')+1}{2}}(\sum_{w\in W_k^l}(-1)^{sgn(w)}
\prod_{i=1}^kT(\nu_{\s'}^{i}\rho',w(i)+l-i))=$$
$$=\nu^{-\frac{k'+l'}{2}-\frac{s(\s')-1}{2}}(\sum_{w\in W_k^l}(-1)^{sgn(w)}
\prod_{i=1}^{ks(\s')}\nu^iT(\rho',w(i)+l-i))=\bar{u}(\s',ks(\s'))$$
which is irreducible.

The formula defining $U$ is an alternated sum of $|W_k^l|$ terms which are distinct elements of $\b'_n$. The term
$\prod_{i=1}^{k}\nu_{\s'}^{i-\frac{k+1}{2}}\s'$,
corresponding to $w$ trivial, is maximal. To prove it, one may use lemma \ref{length} and the fact that
one has ${\bf l}(\prod_{i=1}^{k}\nu_{\s'}^{i-\frac{k+1}{2}}\s')=l$, while ${\bf l}(t)>l$ for any
other term $t$ in the sum.
The Langlands quotient of this maximal term $\prod_{i=1}^{k}\nu_{\s'}^{i-\frac{k+1}{2}}\s'$
is $u'(\s',k)$ and appears then in the sum with coefficient 1.
 So we may write:
$$U=\pi'_0+\sum_{j=1}^m b_j\pi'_j$$
where $\pi'_0=u'(\s',k)$, $b_j$ are non-zero integers, $\pi'_j\in Irr'_n$ and the $\pi'_j$, $0\leq j\leq m$,
 are distinct, with the convention $m=0$ if $U=u'(\s',k)$.
The representation $\nu^{i-\frac{s(\s')}{2}}\pi'_j$ is rigid and
supported on the line generated by $\nu^{i-\frac{s(\s')}{2}}\rho'$.
For different $i$ in $\{1,2,...,s(\s')\}$, these lines are
different. So, as the $\pi'_j$ are distinct (and have distinct
esi-support) $\prod_{i=1}^{s(\s')}\nu^{i-\frac{s-1}{2}}U$ is a
linear combination of exactly $(m+1)^{s(\s')}$ irreducible distinct
representations each appearing with non zero coefficient. As it is
irreducible, we have $m=0$.\qed

\subsection{Transfer of $u(\s,k)$}

Let $k$, $l$, $p$ be positive integers, set $n=klp$ and let $\rho\in\cusp^{u}_p$ and
$\s=Z^u(\rho,l), \in \d^{u}_{lp}$, $\tau=Z^u(\rho,l)\in \d^{u}_{kp}$. Let $s$ be the smallest positive integer such that
$d|sp$. In the next proposition  we give the general result of the transfer of $u(\s,k)$. The question has no meaning
unless $d|n$ (i.e. $s|kl$) which we shall assume.

\begin{prop}\label{transfer}

(a) If $d|lp$ (i.e. $s|l$), then $\s'=\ccc(\s)$ is well defined; we have $s=s(\s')$ and
$$\lj(u(\s,k))=\bar{u}(\s',k).$$

(b) If $d|kp$ (i.e. $s|k$), then $\tau'=\ccc(\tau)$ is well defined; we have $s=s(\tau')$ and
$$\lj(u(\s,k))=\varepsilon |i'(\bar{u}(\tau',l))|$$
where $\varepsilon=1$ if $s$ is odd
and $\varepsilon=(-1)^{\frac{kl}{s}}$ if $s$ is even.

(c) If $d$ does not divide neither $lp$, nor $kp$ (i.e. $s$ does not
divide neither $l$ nor $k$), then  $\lj(u(\s,k))=0$.
\end{prop}

{\bf Proof.} (a) We have the formula of the decomposition of $u(\s,k)$ on the standard basis $\b_n$
(proposition  \ref{formula1}) so
we may compute the formula of the decomposition of $\lj(u(\s,k))$ on the standard basis $\b'_n$ by transfer.
In the meantime, we have the formula of the decomposition of $\bar{u}(\s,k)$ on the standard basis
$\b'_n$ using the formula
\ref{equa1}
and the
corollary  \ref{next} (b). The equality of the two decompositions on the basis $\b'_n$
leads again to the combinatoric lemma 3.1 in [Ta5].

b) Up to the sign $\e$, this is a consequence of the point (a) and the dual transform, theorem \ref{eu} (d),
since $|i(u(\tau,l))|=u(\s,k)$.
For the sign $\e$, see proposition  4.1, b) in [Ba4].

c) The proof is in [Ta6]. It is a consequence of the proposition
\ref{formula1} here, which is also due to Tadi\'c, and the following
lemma for which we give here a more straightforward proof.

\begin{lemme}
Let $k,l,s\in \n^*$. Assume there is a permutation $w$ of $\{1,2,...,k\}$ such that for all $i\in \{1,2,...,k\}$
one has $s|l+w(i)-i$. Then $s|k$ or $s|l$.
\end{lemme}

{\bf Proof.} Let [x] denote the bigger integer less than or equal to x. If $y\in \n$, let $\n_y$ denote the set
$\{1,2,...,y\}$.

Assume $s$ does not divide $l$. Summing up all the $k$ relations
$s|l+w(i)-i$ we find that $s|kl$. So, if $(s,l)=1$, then $s|k$.
Assume $(s,l)=p$. Then for all $i\in \{1,2,...,k\}$, $p|w(i)-i$. Let
$w_0$ be the natural permutation of $\n_{[\frac{k}{p}]}$ induced by
the restriction of $w$ to $\{p,2p,...,[\frac{k}{p}]p\}$ and $w_1$
the natural permutation of $\n_{[\frac{k-1}{p}]+1}$ induced by the
restriction of $w$ to $\{1,p+1,...,  [\frac{k-1}{p}] p+1\}$. Then
for all $i\in \n_{[\frac{k}{p}]}$ one has
$\frac{s}{p}|\frac{l}{p}+w_0(i)-i$, and for all $j\in
\n_{[\frac{k-1}{p}]+1}$ one has $\frac{s}{p}|\frac{l}{p}+w_1(j)-j$.
As now $(\frac{s}{p},\frac{l}{p})=1$ we have seen one has
$\frac{s}{p}|[\frac{k}{p}]$ and $\frac{s}{p}|[\frac{k-1}{p}]+1$.
This implies $[\frac{k}{p}]=[\frac{k-1}{p}]+1$ and so $p|k$. It
follows $\frac{s}{p}|\frac{k}{p}$, i.e. $s|k$.\qed

\subsection{New formulas}\label{newform}

The reader might have noticed that the dual of representations $u(\tau,l)$ and $u'(\tau',l)$ are of the same type,
while the dual of representations $\bar{u}(\tau',l)$ are in general more complicated. This
 is why the point (b) of the proposition  \ref{transfer} looks awkward, we couldn't express
$i'(\bar{u}(\tau',l))$ in terms of $\s'=\ccc(\s)$, and for the good
reason that $\ccc(\s)$ is not defined since the group on which $\s$
lives does not have the good size (divisible by $d$). Recall the
hypothesis was $s(\s')|k$. We explain here that one can get a
formula though, in terms of $u'(\s'_+,\frac{k}{s(\s')})$ and
$u'(\s'_-,\frac{k}{s(\s')})$, where $\s'_+=\ccc(\s_+)$ and
$\s'_-=\ccc(\s_-)$, where the representations $\s_+$ and $\s_-$ are
obtained from $\s$ by stretching it and shortening it as to have
good size to transfer. The formulas we will give here are required
for the global proofs.

Let $\tau'\in\d'_n$ and
$l=as(\s')+b$ with $a,b\in\n$, $1\leq b\leq s(\s')-1$. We start with  the formula \ref{equa1}:

$$\bar{u}(\tau',l)=\prod_{i=1}^b \nu^{i-\frac{b+1}{2}} u'(\tau',a+1)\times \prod_{j=1}^{s(\s')-b}
 \nu^{j-\frac{s(\tau')-b+1}{2}}u'(\tau',a).$$
So one may compute the dual of $\bar{u}(\tau',l)$ using
  proposition  \ref{dual}; if $\tau'=T^u(\rho',k)$,  we set $\s'_+=T^u(\rho',a+1)$
  and, if $a\neq 0$, $\s'_-=T^u(\rho',a)$; then
\begin{equation}\label{formulainvolution}
|i'(\bar{u}(\tau',l))|=\prod_{i=1}^b \nu^{i-\frac{b+1}{2}} u'(\s'_+,l)\times
\prod_{j=1}^{s(\tau')-b}
 \nu^{j-\frac{s(\tau')-b+1}{2}}u'(\s'_-,l)
\end{equation}
with the convention that if $a=0$ we make abstraction of the second product.

In particular the dual of a representation of type $\bar{u}(\s',k)$
is of the same type (i.e. some $\bar{u}(\gamma,p)$) if and only if
 $s(\s')|k$ or $\s'$ is cuspidal
and $k<s(\s')$. One can see it
comparing the formula \ref{formulainvolution} with the
formula \ref{equa2} and using the fact that a  product of
representations of the type $\nu^\a u'(\s',k)$ determines its factors up to permutation ([Ta2]).

This gives a formula for $\lj(u(\s,k))$ when $s$ divides $k$ but $s$
does not divide $l$ (case (b) of the proposition  \ref{transfer}).
Let $|\lj|(\s,k)$ stand for the one irreducible representation in
$\{\lj(u(\s,k)),-\lj(u(\s,k))\}$. Let $\rho\in \cusp_p^u$,
$\s=Z^u(\rho,l) \in \d_{lp}^u$ and let $s$ be the smallest positive
integer such that $d|ps$. Assume $s\neq 1$ and $l=as+b$, $a,b\in\n$,
$1\leq b\leq s-1$. Set $\s_+=Z^u(\rho,(a+1)s)$ and, if $a\neq 0$,
$\s_-=Z^u(\rho,as)$. Let $\s'_+=\ccc(\s_+)$ and, if $a\neq 0$,
$\s'_-=\ccc(\s_-)$. If $s|k$, if $k=k's$, then

\begin{equation}\label{product}
|\lj|(u(\s,k))=\prod_{i=1}^b \nu^{i-\frac{b+1}{2}}  u'(\s'_+,k')\times
\prod_{j=1}^{s(\s')-b}
 \nu^{j-\frac{s(\s')-b+1}{2}}u'(\s'_-,k'),
\end{equation}
with the convention that if $a=0$ we ignore the second product.

The following formula for the transfer is somehow artificial, but it
has the advantage of being symmetric in $k$ and $l$ and adapted to
both the cases (a) and (b) of the proposition \ref{transfer}. Let
$\rho\in \cusp_p$ for some $p\in \n^*$, and let $s$ be the smallest
positive integer such that $d|ps$. Set $\rho'=\ccc(Z^u(\rho,s))$ (in
particular $\rho'$ is cuspidal and $s(\rho')=s$). Let $k,l\in \n^*$.
Set $b=k-s[\frac{k}{s}]+l-s[\frac{l}{s}]$ and define a sign
$\varepsilon$ by $\varepsilon=1$ if $s$ is odd and
$\varepsilon=(-1)^{\frac{kl}{s}}$ if $s$ is even. Make the
convention that a product $\prod_{i=1}^0$ has to be ignored in a
formula. The representation $u(Z^u(\rho,l),k)$ is $d$-compatible if
and only if $s|k$ or $s|l$. In this case we have

\begin{equation}\label{beta}
\lj(u(Z^u(\rho,l),k))=\varepsilon
\prod_{i=1}^b \nu^{i-\frac{b+1}{2}}  u'\big( T^u(\rho',\bigg[ \frac{l}{s}\bigg] ),\bigg[ \frac{k-1}{s}\bigg] +1\big)
\end{equation}
$$
\times
\prod_{j=1}^{s-b}
 \nu^{j-\frac{s-b+1}{2}}u'\big( T^u(\rho',\bigg[ \frac{l-1}{s}\bigg] +1),\bigg[ \frac{k}{s}\bigg] \big),
$$
with the convention that in this formula we ignore the first product
if $[\frac{l}{s}]=0$ and the second product if $[\frac{k}{s}]=0$.
(As $s$ divides either $l$ or $k$ we cannot have
$[\frac{l}{s}]=[\frac{k}{s}]=0$.)

\subsection{Transfer of unitary representations}\label{unit0}
Let ${\mathcal U'}$ be the set of all the representations
$u'(\s',k)$ and $\pi(u'(\s',k),\a)$ where $k,l$ range over $\n^*$,
$\s'\in \cusp'_l$ and $\a\in ]0,\frac{1}{2}[$. Here we will use the
fact that the representations $u'(\s',k)$ are unitary so we will
assume the characteristic of the base field $F$ is zero. As Henniart
pointed out to me it is not difficult to lift the result to the non
zero characteristic case by the close fields theory, but it has not
been written yet.

The next proposition has been proven in [Ta6] under the assumption
of the $U_0$ conjecture of Tadi\'c. We prove it here without this
assumption.

\begin{prop}\label{unit}
(a) All the representations in $\mathcal U'$ are irreducible and unitary.

(b) If $\pi'_i\in {\mathcal U'}$, $i\in\k$, then the product
$\prod\ki \pi'_i$ is irreducible and unitary.

(c) If $u\in Irr^u_{nd}$, then $\lj(u)=0$ or $\lj(u)$ is an
irreducible unitary representation $u'$ of $G'_n$ up to a sign.

(d) Let $u'$ be an irreducible unitary representation of $G'_n$. If $u'\times u'$ is irreducible, then
$u'$ is a product of representations in $\mathcal U'$.
\end{prop}
\ \\

The point (a) is part of the Tadi\'c conjecture U2 in [Ta2].
It has already been solved for $s(\s')\geq 3$ in [BR1], remark 4.3,
 which is actually a remark due to Tadi\'c, not to the authors. The only problem, as explained in [Ta2], is
to show that the product $u'(\s',k)\times u'(\s',k)$ is irreducible.  This is just our corollary  \ref{next} (a).

(b) This follows from the irreducibility trick (the proposition  \ref{trick}) and the corollary  \ref{next} (a).

(c) This is a consequence of the proposition  \ref{transfer}, the formula \ref{product}
and of the points (a) and (b) here above.

(d) Suppose by absurd $u'\times u'$ is irreducible. Then any product containing $u'$ and representations in
$\mathcal U'$ is irreducible (by proposition  \ref{trick}). As $u'(\s',k)$ are prime elements
([Ta2], 6.2), the same proof as for $GL(n)$ (Tadi\'c, [Ta1]) shows that
$u'$ is a product of representations in $\mathcal U'$.\qed
\ \\

If $u'$ is like in the second situation of the point (c) we write
$u'=|\lj^u|(u)$.

Let $\Pi{\mathcal U'}$ be the set of products of representations in $\mathcal U'$.
Then $\Pi{\mathcal U'}$ is a set of irreducible unitary representations containing the $\bar{u}(\s',k)$
(formula \ref{equa1}). We have:

\begin{prop}\label{stable}
(a) The set $\Pi{\mathcal U'}$ is stable under $|i'|$.

(b) If $\pi$ is a $d$-compatible unitary representation of $G_{nd}$,
then $|\lj^u|(\pi)\in \Pi{\mathcal U'}$.
\end{prop}

{\bf Proof.} (a) is implied by proposition  \ref{br} (a).

(b) is implied by proposition  \ref{transfer}, the fact that $\bar{u}(\s',k)\in \Pi{\mathcal U'}$ and the point (a).\qed
\ \\

So we have a map $|\lj^u|$ from the set of unitary irreducible
$d$-compatible representations of $G_{nd}$ to the set $\Pi{\mathcal
U'}$. We prove here a lemma we will need further to construct
automorphic unitary representations of the inner form which do not
transfer to the split form.

\begin{lemme}\label{nonunit}
Assume $dim_FD=16$. Let $St'_3$ be the Steinberg representation of $GL_3(D)$ and $St'_4$ the Steinberg representation
of $GL_4(D)$. Let

$$\pi=\nu^{-\frac{3}{2}}u'(St'_3,4)\times\nu^{-\frac{1}{2}} u'(St'_4,3)\times\nu^{\frac{1}{2}} u'(St'_4,3)\times
\nu^{\frac{3}{2}}u'(St'_3,4).$$
Then

(i) $\pi$ is unitary,

(ii) we have $\pi<\bar{u}(St'_3,16)$ and

(iii) $\pi$ is not in the image of $|\lj^u|$.
\end{lemme}

{\bf Proof.} (i) If $1_1$ is the trivial representation of
$D^\times$, we have $s(1_1)=4$. So $s(St'_3)=s(St'_4)=4$. By
definition of $\Pi {\mathcal U}'$    it is clear then that $\pi\in
\Pi {\mathcal U}'$.

(ii) By the formula \ref{equa2} we get

$$\bar{u}(St'_3,16)=
\nu^{-\frac{3}{2}}u'(St'_3,4)\times\nu^{-\frac{1}{2}} u'(St'_3,4)\times\nu^{\frac{1}{2}} u'(St'_3,4)\times
\nu^{\frac{3}{2}}u'(St'_3,4).$$

It is easy to prove that the esi-support of  $u'(St'_4,3)$ is
obtained from the esi-support of $u'(St'_3,4)$ by elementary
operations. So $\pi<\bar{u}(St'_3,16)$

(iii) Any unitary representation of $G_{nd}$ decomposes in a unique
way up to permutation of factors in a product of representations of
type $\nu^\a u(\s,k)$ and any unitary representation of $G_{nd}$
decomposes in a unique way up to permutation of factors in a product
of representations of type $\nu^\a u'(\s',k)$ ([Ta2]).
 The formula \ref{beta}
implies that if $\nu^{-\frac{3}{2}}u'(St'_3,4)$ appear in the
decomposition of an element of the image of $|\lj^u|$, then
$\nu^{-\frac{1}{2}}u'(St'_3,4)$ should appear too. So $\pi$ is not
in the image of $|\lj_u|$. \qed
\ \\

It is natural to ask how many antecedents has a given element $u'\in
\Pi{\mathcal U'}$. A product of representations of type
$\bar{u}(\s',k)$ and $|i'|\bar{u}(\s',k)$ may be equal to several
different similar products and it does not seem to exist a
manageable formula for the number of possibilities. They are of
course finite since the cuspidal support is fixed.

\subsection{Transfer of local components of global discrete series}\label{generic}
Let $\gamma\in Irr^u_n$ be a generic representation. Then one may
write
$$\gamma=\prod_{i=1}^m\nu^{e_i}\s_i$$
where $\s_i$ are square integrable and $e_i\in ]-\frac{1}{2},\frac{1}{2}[$ ([Ze]). As
    it is explained in the section 4.1 of [Ba4], for all $k\in\n$, the
representation $\prod_{i=0}^{k-1}(\nu^{\frac{k-1}{2}-i}\gamma)$ is a
standard representation and if we call
$Lg(\gamma,k)$ its Langlands quotient,
then we have
$$Lg(\gamma,k)=\prod_{i=1}^m\nu^{e_i}u(\s_i,k).$$
One may show that, as $\gamma$ was unitary, $Lg(\gamma,k)$ is unitary.
$\gamma$ is tempered if and only if all $e_i$ are zero.
As the local component of global
cuspidal representations are generic (see the next section),
by the Moeglin-Waldspurger classification
all local component of global discrete series of $GL_n$ are of the
type $Lg(\gamma,k)$, so it is important to know when do they transfer
to a non zero representation under $\lj$.

Write
$\s_i=Z^u(\rho_i,l_i)$, $\rho_i\in\cusp^u_{p_i}$.
Let $J$ be the set of integers $j\in\{1,2,...,m\}$ such that
$d|p_jl_j$.
Let $s_{\gamma,d}$ be the smallest positive integer $s$ such that for all
$i\in\{1,2,...,m\}\backslash J$, $d|p_is$. Then
$\lj(Lg(\gamma,k))\neq
0$ if and only if for all
$i\in\{1,2,...,m\}$ we have $\lj(u(\s_i,k))\neq 0$
if and only if $s_{\gamma,d}|k$ (by proposition  \ref{transfer}). Then
$$\lj(Lg(\gamma,k))=\prod_{i=1}^m\nu^{e_i}\lj(u(\s_i,k)).$$

\section{Basic facts and notations (global)}

Let $F$ be a global field {\it of characteristic zero} and $D$ a
central division algebra over $F$ of dimension $d^2$. Let
$n\in\n^*$. Set $A=M_{n}(D)$. For each place $v$ of $F$ let $F_v$ be
the completion of $F$ at $v$ and set $A_v=A\otimes F_v$. For every
place $v$ of $F$, $A_v\simeq M_{r_v}(D_v)$ for some positive number
$r_v$ and some central division algebra $D_v$ of dimension $d_v^2$
over $F_v$ such that $r_v d_v=nd$. We will fix once and for all an
isomorphism and identify these two algebras. We say that $M_n(D)$
{\bf is split} at the place $v$ if $d_v=1$. The set $V$ of places
where $M_n(D)$ is not split is finite. We assume in the sequel {\it
$V$ does not contain any infinite place}. For each $v$, $d_v$
divides $d$, and moreover $d$ is the smallest common multiple of the
$d_v$ over all the places $v$.

Let $G'(F)$ be the group $A^\times=GL_n(D)$. For every place $v\in
V$, set $G'_v= A_v^\times= GL_{r_v}(D_v)$. For every finite place
$v$ of $F$, we set $K_v=GL_{r_v}(O_v)$, where $O_v$ is the ring of
integers of $D_v$. We fix then a Haar measure $dg_v$ on $G'_v$ such
that $vol(K_v)=1$. For every infinite place $v$, we fix an arbitrary
Haar measure $dg_v$ on $G'_v$. Let $\aa$ be the ring of ad\`eles of
$F$. With these conventions, the group $G'(\aa)$ of ad\`eles of
$G'(F)$ is the restricted product of the $G'_v$ with respect to the
family of compact subgroups $K_v$. We consider the Haar measure $dg$
on $G'(\aa)$ which is the restricted product of the measures $dg_v$
(see [RV] for details). We see $G'(F)$ as a subgroup of $G'(\aa)$
via the diagonal embedding.

\subsection{Discrete series}

Let $Z(F)$ be the center of $G'(F)$. For every place $v$, let $Z_v$
be the center of $G'_v$. For every finite place $v$ of $F$, let
$dz_v$ be a Haar measure on $Z_v$ such that the volume of $Z_v\cap
K_v$ is one. The center $Z(\aa)$ of $G'(\aa)$ is canonically
isomorphic the restricted product of the $Z_v$ with respect to the
$Z_v\cap K_v$. On $Z(\aa)$ we fix the Haar measure $dz$ which is the
restricted product of the measures $dz_v$. On $Z(\aa)\bc G'(\aa)$ we
consider the quotient measure $dz\bc dg$. As $G'(F)\cap Z(\aa)\bc
G'(F)$ is a discrete subgroup of $Z(\aa)\bc G'(\aa)$, on the
quotient space $Z(\aa)G'(F)\bc G'(\aa)$ we have a well defined
measure coming from $dz\bc dg$. The measure of the whole space
$Z(\aa)G'(F)\bc G'(\aa)$ is finite.

Through all these identifications, $Z(F)$ is a subgroup of $Z(\aa)$.
Fix a unitary smooth character $\o$ of $Z(\aa)$, trivial on $Z(F)$.

Let $\lgg$ be the space of functions $f$ defined on $G'(\aa)$ with values in $\cc$ such that

i) $f$ is left invariant under $G'(F)$,

ii) $f$ verify $f(zg)=\o(z)f(g)$ for all $z\in Z(\aa)$ and all $g\in G'(\aa)$,

iii) $|f|^2$ is integrable over $Z(\aa)G'(F)\bc G'(\aa)$.\\

We consider the representation $R'_\o$ of $G'(\aa)$ by right
translations in the space $\lgg$. We call {\bf discrete series of}
$G'(\aa)$ any irreducible subrepresentation of any representation
$R'_\o$ for any unitary smooth character $\o$ of $Z(\aa)$ trivial on
$Z(F)$.

Every discrete series of $G'(\aa)$ with central character $\o$
appears in $R'_\o$ with a finite multiplicity. Every discrete series
$\pi$ of $G'(\aa)$ is isomorphic with a restricted Hilbertian tensor
product of (smooth) irreducible unitary representations $\pi_v$ of
the groups $G'_v$ like in [Fl1]. Each representation $\pi_v$ is
determined by $\pi$ up to isomorphism and is called the {\bf local
component of $\pi$ at the place $v$}. For almost all finite place
$v$, $\pi_v$ has a non zero fixed vector under $K_v$. We say then
$\pi_v$ is {\bf spherical}. In general, an admissible irreducible
representation $\s$ of $G'(\aa)$ decomposes similarly into a
restricted tensor product of smooth irreducible representations
$\s_v$ of $G'_v$ and $\s_v$ is spherical for almost all $v$ (see
[Fl1]).


Let $R'_{\o,disc}$ be the subrepresentation of $R'_\o$ generated by
the discrete series.  If $\pi$ is a discrete series we call the {\bf
multiplicity of $\pi$ in the discrete spectrum} the multiplicity
with which $\pi$ appear in $R'_{\o,disc}$.

\subsection{Cuspidal representations}

Let $\lgg_c$ be the subspace of all the functions $f$ in $\lgg$ verifying
$$\int_{N(F)\bc N(\aa)}f(ng)dn=0$$
for almost all $g\in G'(\aa)$ and for all unipotent radical $N$ of a
parabolic $F$-subgroup of $G'(F)$.

The space $\lgg_c$ is stable under $R'_\o$ and decomposes discretely
in a direct sum of irreducible representations. Such an irreducible
subrepresentation is called {\bf cuspidal}. It is automatically a
discrete series.

We let now $n$ vary. For all $n\in \n^*$ let $G'_n$ be the group of ad\`eles
of $GL_n(D)$ and $G'_{n,v}$ the local component of $G'_n$ at a place $v$.
Let $DS'_n$ be the set of discrete series of $G'_n$.

If $(n_1,n_2, ..., n_k)$ is an ordered set of positive integers such that $n_1+n_2+...+n_k=n$,
we call {\bf standard Levi subgroup} of $G'(F)$ a subgroup formed by diagonal matrices
by blocks of given sizes $n_1$, $n_2$, ..., $n_k$ in this order.

A {\bf standard Levi subgroup} of $G'_n(\aa)$ will be by definition
a subgroup defined by the ad\`ele group $L(\aa)$ of a standard Levi
subgroup $L$ of $G'(F)$. Let $L$ be like in the previous paragraph.
For every place $v$ of $F$, one has $d_v|n_i$ for all $1\leq i\leq
k$. If $L_v$ is the subgroup of $G'_v$ formed by diagonal matrices
by $k$ blocks of sizes $n_1/d_v$, $n_2/d_v$, ..., $n_k/d_v$ in this
order, then $L(\aa)$ is the restricted product of the $L_v$ with
respect to $L_v\cap K_v$.  We naturally identify $L$ with the
ordered product $\times_{i=1}^k G'_{n_i}$.

Let $\nu$ denote here the character $|\det|_F$ on $G'_n$, product of local characters $\nu_v=|\det|_v$ where
$|\ \ \ |_v$ is the normalized absolute value on $F_v$.

\subsection{Automorphic representations}

Let us recall some facts from [La]. Let $L=\times_{i=1}^k G'_{n_i}$
be a standard Levi subgroup of $G'_n$. For $1\leq i\leq k$, let
$\rho_i$ be a cuspidal representation of $G'_{n_i}(\aa)$ and $e_i$ a
real number. Set $\rho=\otimes_{i=1}^{k}\nu^{e_i}\rho_i$.

Then for each place $v$, the induced representation
$\Pi_v=ind_{L_v}^{G'_v}\rho_{v}$ is of finite length. For every
place $v$ where all the $\rho_{i,v}$ are spherical, $\Pi_v$ has a
unique subquotient $\pi_v$ which is a spherical representation. An
irreducible subquotient of $ind_{L(\aa)}^{G'_n(\aa)}\rho$ is said to
be a {\bf constituent} of $ind_{L(\aa)}^{G'_n(\aa)}\rho$. Then an
irreducible admissible representation $\s$ of $G'_n$ is
 a constituent of $ind_L^{G'_n}\rho$ if and only if for all $v$,
$\s_v$ is an irreducible subquotient of $\Pi_v$ and for almost all $v$, $\s_v=\pi_v$.
The notion of cuspidal representation
differs between [La] and here: here we allow only what would be in the [La] language {\it unitary}
 cuspidal representations. Using the proposition  2 in [La],
 an {\bf automorphic} representation $\mathcal A$ of $G'_n$ will be here
by definition a constituent of  $ind_{L(\aa)}^{G'_n(\aa)}\rho$ for
some $\rho$ as before. One would like to prove then that the couples
$(\rho_i,e_i)$ are all determined by $\mathcal A$ up to permutation.
This has been shown in [JS] in the case where $D=F$, and in the
present paper we will show it for general $D$. For the case $D=F$,
we will then call the non ordered multiset
$\{\nu^{e_1}\rho_1,\nu^{e_2}\rho_2,..., \nu^{e_k}\rho_k\}$ the
cuspidal support of $\mathcal A$. For the traditional definition of
automorphic representations we send to [BJ]; here we used an
equivalent one, cf. proposition  2 in [La]. Let us point out that a
discrete series is always a (unitary) automorphic representation.

Some other facts are known in the case $D=F$. But let us use the
following convention: we keep a general division algebra $D$ and the
notation adopted before. And we consider a second class of groups
$G_n=GL_n(F)$, i.e. particular case $D=F$ of what we have seen
before. All the definition adapt then to $G_n$, and we will write
$DS_n$ for the set of discrete series of $G_n(\aa)$.

\subsection{Multiplicity one theorems for $G_n$}

We recall in this subsection three facts about $G_n$. There is the
{\it multiplicity one theorem}: every discrete series of $G_n(\aa)$
appears with multiplicity one in the discrete spectrum. And the {\it
strong multiplicity one theorem}: if $\pi$ and $\pi'$ are two
discrete series of $G_n$ such that $\pi_v=\pi'_v$ for almost all
place $v$, then $\pi=\pi'$. This results may be found in [Sh] and
[P-S] (when $D=F$). We will prove them in this paper for general
$G'_n$.

One also knows that the local component of a cuspidal representation
of $G_n$ at any place is generic and unitary, hence an irreducible
product $\prod_{i=1}^m\nu^{e_i}\s_i$ where $\s_i$ are square
integrable representations and $e_i\in ]-\frac{1}{2},\frac{1}{2}[$
(see [Sh] and [Ze]).

\subsection{The residual spectrum of $G_n$}

We recall now the Moeglin-Waldspurger classification of discrete
series for groups $G_n(\aa)$. Let $m\in \n^*$ and $\rho\in DS_m$ be
a cuspidal representation. If $k\in\n^*$, then the induced
representation $\prod_{i=0}^{k-1}(\nu^{\frac{k-1}{2}-i}\rho)$ has a
unique constituent $\pi$ which is a discrete series (i.e. $\pi\in
DS_{mk}$). One has $\pi_v=Lg(\rho_v,k)$ for all place $v$ (we used
the definition of $Lg(\rho_v,k)$ of the section \ref{generic} since
$\rho_v$ is generic). We set then $\pi=MW(\rho,k)$. Discrete series
$\pi$ of groups $G_n(\aa)$, $n\in\n^*$, are all of this type, $k$
and $\rho$ are determined by $\pi$ and $\pi$ is cuspidal if and only
if $k=1$. These are results in [MW2]. We will prove further the same
classification holds for $G'_n(\aa)$

Let us prove, for further purposes, a proposition  generalizing the
strong multiplicity one theorem.

\begin{prop}\label{multone}
Let $\s_i\in DS_{n_i}$, $i\in\{1,2,...,k_1\}$, $\sum_{i=1}^{k_1} n_i=n$ and
$\tau_j\in DS_{m_j}$, $j\in\{1,2,...,k_2\}$, $\sum_{j=1}^{k_2}
m_j=n$.
Assume that for almost all finite places $v$ the local components of
the (irreducible) products $\s=\prod_{i=1}^{k_1}\s_i$ and
$\tau=\prod_{j=1}^{k_2}\tau_j$ at the place $v$ are equal. Then
$(\s_1,\s_2,...,\s_{k_1})$ equals up to a permutation
$(\tau_1,\tau_2,...,\tau_{k_2})$.
\end{prop}

{\bf Proof.} By the theorem 4.4 in [JS], the cuspidal supports of
the automorphic representations $\s$ and $\tau$ are equal. We call a
{\bf line} the set of representations $\{\nu^k\rho\}_{k\in \z}$,
where $\rho$ is a cuspidal representation of some $G_m(\aa)$. We
call a {\bf shifted line}  the set of representations
$\{\nu^{k+\frac{1}{2}}\rho\}_{k\in \z}$, where $\rho$ is a cuspidal
representation of some $G_m(\aa)$. Thanks to the Moeglin-Waldspurger
classification we know that the set of the elements of the cuspidal
support of a given $\s_i$ or $\tau_j$ is either included in a line,
or in a shifted line. So we may then ``separate the supports'' and
reduce the problem to the case where there exists a line or a
shifted line $T$ such that the set of elements of the cuspidal
supports of all the $\s_i$ and all the $\tau_j$ are included in $T$.
Then there exists a cuspidal representation $\rho$ such that
$\s_i=MW(\rho,p_i)$ for all $i$ and $\tau_j=MW(\rho,q_j)$ for all
$j$. And moreover the $p_i$ and the $q_j$ are either all odd, or all
even. Let $X$ be the cuspidal support of $\s$ and $\tau$ in this
case. We show that $X$ determines the $\s_i$ up to permutation.

If the $p_i$ are all odd, the result is a consequence of the following
combinatorial lemma:
\begin{lemme}
Let $A$ be a multiset of integer numbers which may
 be written as a
reunion with multiplicities of sets of the form
$B=\{-k,-k+1,-k+2,...,k-2,k-1,k\}$. Then the sets $B$ are determined by
$A$.
\end{lemme}

{\bf Proof.} Let $f:\z\to \n$ be the multiplicity map: $f(a)$ is the
multiplicity of $a$ in $A$. The number $f(a)$ is also the number of
sets $B$ containing $a$. If $a\geq 1$ if a set contains $a$ it
contains also $a-1$. So $f$ is decreasing on $\n$ and for all $p\in
\n$, the number of sets  $\{-p,-p+1,-p+2,...,p-2,p-1,p\}$ in $A$ is
exactly $f(p)-f(p+1)$.\qed

If the $p_i$ are even, the proof is essentially the same.
This finishes the proof of the proposition  \ref{multone}.\qed

\subsection{Transfer of functions}

For each finite place $v$ let $H(G'_{n,v})$ be the Hecke algebra of
locally constant functions with compact support on $G'_{n,v}$. Let
$H(G'_n)$ be the set of functions $f:G'_n(\aa)\to\cc$ such that $f$
is a product $f=\prod_vf_v$ over all places of $F$, such that $f_v$
is $C^\infty$ with compact support when $v$ is infinite, $f_v\in
H(G'_{n,v})$ when $v$ is finite and, for almost all finite place
$v$, $f_v$ is the characteristic function of $K_v$. We write then
$f=(f_v)_v$. As the local components of an automorphic
representation $\pi$ are almost all spherical, the product
$\prod_v\tr\pi_v(f_v)$ has a meaning for all $f=(f_v)_v\in H(G'_n)$
and we may set $\tr(\pi(f))=\prod_v\tr\pi_v(f_v)$. We adopt similar
notations and definitions for the groups $G_n$.

Let $v\in V$. We fix measures on maximal tori of $G_{nd,v}$ and
$G'_{n,v}$ in a compatible way and define orbital integrals $\Phi$
on $G_{nd,v}$ and $\Phi'$ on $G'_{n,v}$ for regular semisimple
elements with respect to these choices (see the section 2 of [Ba1]
for example). If
 $f_v\in H(G_{nd,v})$ and $f'_v\in H(G'_{n,v})$ we say that $f_v$ and $f'_v$
{\bf correspond} to each-other, and write  $f_v\lra f'_v$, if:

-  $f_v$ and $f'_v$ are supported in the set of regular semisimple
elements and

- for all $g\lra g'$ we have $\Phi(f_v,g)=\Phi'(f'_v,g')$ and

- for all regular semisimple $g\in G_{nd,v}$ which does not
correspond to any $g'\in G'_{n,v}$ we have $\Phi(f_v,g)=0$.\\
It is
known that for every $f'_v\in H(G'_{n,v})$ supported on the regular
semisimple set there exists $f_v\in H(G_{nd,v})$ such that $f_v\lra
f'_v$. Also, if $f_v\lra f'_v$ then $tr(\pi(f_v))=0$ for all
representation $\pi$ induced from a Levi subgroup of $G_{nd,v}$
which does not transfer (section 2 of [Ba1] for example).

For $f=(f_v)_v\in H(G_{nd})$ and $f'=(f'_v)_v\in H(G'_n)$ we write $f\lra f'$ and say that $f$ and $f'$
{\bf correspond} to each other if

i) $\forall v\notin V$ we have $f_v=f'_v$ and

ii) $\forall v\in V$ we have $f_v\lra f'_v$.\\
For every $f'=(f'_v)_v\in H(G'_n)$ such that for all $v\in V$ the
support of $f'_v$ is included in the set of regular semisimple
elements of $G'_v$ there exists $f\in H(G_n)$ such that $f\lra f'$.
If $f\in H(G_{nd})$, we say $f$ {\bf transfers} if there exists
$f'\in H(G'_n)$ such that $f\lra f'$.

\section{Global results}

\subsection{Global Jacquet-Langlands, multiplicity one and strong multiplicity one for inner forms}

\def\g{{\bf G}}

\def\tr{{\rm tr}}

For all $v\in V$, denote $\lj_v$ (resp. $|\lj|_v$) the
correspondence $\lj$ (resp. $|\lj|$), as defined at the subsection
\ref{extended}, applied to $G_{nd,v}$ and $G'_{n,v}$.

If $\pi\in SD_{nd}$ we say $\pi$ is $D${\bf -compatible} if, for all
$v\in V$, $\pi_v$ is $d_v$-compatible. Then $\lj(\pi_v)\neq 0$ and
$|\lj|_v(\pi_v)$ is an irreducible representation of $G'_{n}$
(proposition  \ref{main} (c)).

\def\f{{\bf G}}

\begin{theo}\label{correspondence}
(a) There exist a unique map $\f:DS'_{n}\to DS_{nd}$ such that for
all $\pi'\in DS'_n$, if $\pi=\f(\pi')$, one has
$|\lj|_v(\pi_v)=\pi'_v$ for all place $v\in V$, and $\pi_v=\pi'_v$
for all place $v\notin V$. The map $\f$ is injective. The image of
$\f$ is the set $DS^D_{nd}$ of $D$-compatible discrete series of
$G_{nd}(\aa)$.

(b) We have multiplicity one theorem for discrete series of
$G'_n(\aa)$: if $\pi'\in DS'_n$, then the multiplicity of $\pi'$ in
the discrete spectrum is one.

(c) We have strong multiplicity one theorem for discrete series of
$G'_n(\aa)$: if $\pi', \pi''\in DS'_n$, if $\pi'_v=\pi''_v$ for
almost all $v$, then $\pi'_v=\pi''_v$ for all $v$.

(d) For all $\pi'\in DS'_n$, for all place $v\in V$, $\pi'_v\in \Pi{\mathcal U'}$  (see
section  \ref{unit0}).
\end{theo}

{\bf Proof.} We will use the results in [AC]. The authors compare the trace formulas of $G_{nd}$ and of $G'_n$.
We will restate the result here.

Let $F^*_\infty$ be the product $\times_i F^*_i$ where $i$ runs over the set of infinite places of $F$. Let $\mu$ be
a unitary character of $F^*_\infty$. We use the embedding of $F^*_\infty$ in $\aa^\times$ trivial at finite places
to realize it as a subgroup of the center $Z(\aa)$.

Let ${\mathcal L}(G_{nd})$ be the set of $F$-Levi subgroups of $G_{nd}$ which contains the maximal diagonal torus.

Let
$$I_{disc,t,\mu,G_{nd}}(f)=$$
$$\sum_{L\in {\mathcal L}(G_{nd})} |W_0^L||W_0^{G_{nd}}|^{-1}\sum_{s\in W({\frak a}_L)_{reg}}
|\det(s-1)_{{\frak a}_L^{G_{nd}} }|^{-1}
\tr(M_L^{G_{nd}}(s,0)\rho_{L,t}(0,f))$$ where, in order of the
apparition:

- $t\in \r_+$;

- $|W_0^L|$ is the cardinality of the Weyl group of $L$;

- $|W_0^{G_{nd}}|$ is the cardinality of the Weyl group of $G_{nd}$;

- ${\frak a}_L$ is the real space $Hom(X(L)_F,\r)$ where $X(L)_F$ is
the lattice of rational characters of $L$; $W({\frak a}_L)$ is the
Weyl group of ${\frak a}_L$ of $L$; ${\frak a}_L^{G_{nd}}$ is the
quotient of ${\frak a}_L$ by ${\frak a}_{G_{nd}}$; $W({\frak
a}_L)_{reg}$ is the set of $s\in W({\frak a}_L)$ such that
$\det(s-1)_{{\frak a}_L^{G_{nd}}}\neq 0$;

- $M(s,0)$ it the intertwining operator associated to $s$ at the point $0$; it intertwines representations
$ind_L^{G_{nd}}\s$ and $ind_{sL}^{G_{nd}}s\s$, where $\s$ is a representation of $L$;

- $\rho_{L,t}$ is the induced representation with respect to any parabolic subgroup
with Levi factor $L$ from the direct sum of discrete series $\pi$
of $L$ such that $\pi$ is $\mu$-equivariant
and the imaginary part of the Archimedean infinitesimal character of $\pi$
has norm $t$ ([AC], page 131-132);

- $f$ is an element of $H(G_{nd})$.

For this definition see page 198, and the formula (4.1) page 203,
in [AC].  It is the ``$\mu$ formula'', and not the original
definition-equality (9.2) page 132, which does not contain any $\mu$.

Now let us compute the terms. It turns out that $W({\frak
a}_L)_{reg}$ is empty unless $L$ is conjugated to a Levi subgroup
given by matrices by blocks of equal size. Let $L$ be the Levi
subgroup given by diagonal matrices by $l$ blocks of size $m$,
$lm=nd$. If we identify $W({\frak a}_L)$ with $\frak{S}_l$, then
$W({\frak a}_L)_{reg}$ is the set of $l$-cycles. So the cardinality
of $W({\frak a}_L)_{reg}$ is $(l-1)!$ and for any $s\in W({\frak
a}_L)_{reg}$, $|\det(s-1)_{{\frak a}_L^{G_{nd}} }|=l$. We also have
$|W_0^L|=(m!)^l$ and $|W_0^{G_{nd}}|=(nd)!$. So the coefficient of
the character attached to $L$ in the linear combination over
${\mathcal L}(G_{nd})$
 is
$\frac{(m!)^l}{(nd)!\, l}$. Now, if $L'$ is conjugated with $L$, the contribution of $L'$ to the sum
is the same as that of $L$ ([AC], page 207).
Let us compute the number of Levi subgroups $L'$ conjugated to $L$, and containing
the diagonal torus. The diagonal torus is then a maximal torus of $L'$, and so the center of $L'$ is contained
in the diagonal torus. As $L'$ is the centralizer of its center there will be exactly as many $L'$
as the non ordered partitions of $\{1,2,...,nd\}$ in $l$ subsets of cardinality $m$.
This number is $l!^{-1}C_{nd}^{nd-m}
C_{nd-m}^{nd-2m}C_{nd-2m}^{nd-3m}...C_{2m}^m$, which is $\frac{(nd)!}{l!(m!)^l}$ (for
 a more theoretical formula for the same result see [AC], page 207).

So we may rewrite the formula: if for all $l|nd$, $L_l$ is the Levi subgroup of $G_{nd}$ given by
matrices by $l$ blocks of equal size $\frac{nd}{l}$, then
$$I_{disc,t,\mu,G_{nd}}(f)=$$
$$\sum_{l|nd}\frac{1}{l!\, l}\sum_{s\in W({\frak a}_L)_{reg}}
 \tr(M_{L_l}^{G_{nd}}(s,0)\rho_{L_l,t}(0,f)).$$

In [AC] it is shown moreover, page 207-208, that for any $L_l$, the $(l-1)!$ elements $s\in W({\frak a}_L)_{reg}$ give
all the same contribution to the sum. So, in the end, if $s_0$ is the cycle $(1,2,...,l)$, the definition of
$I_{disc,t,\mu,G_{nd}}(f)$ turns out to be simply:

$$\sum_{l|nd}\frac{1}{l^2}
 \tr(M_{L_l}^{G_{nd}}(s_0,0)\rho_{L_l,t}(0,f)).$$

Let us turn now to the operator $M_{L_l}^{G_{nd}}(s_0,0)\rho_{L_l,t}(0,f)$. A discrete series $\rho$ of $L_l$ is an
ordered product $\otimes_{i=1}^l \rho_i$, where each $\rho_i$ is a discrete series of $G_{\frac{nd}{l}}$.
Let $Stab_\rho$ be the subgroup of $\frak{S}_l$ which stabilizes the ordered multiset $(\rho_1,\rho_2,...,\rho_l)$
 for the obvious action.
Let $X_\rho$ be a set of representatives of $\frak{S}_l/Stab_\rho$ in $\frak{S}_l$.
Let $V_\rho$ be the subspace $\oplus_{x\in X_\rho} \times_{i=1}^l \rho_{x(i)}$
of $\rho_{L_l,t}$. Then $V_\rho$ is stable under the operator $M_{L_l}^{G_{nd}}(s_0,0)$. But, if the $\rho_i$ are not
all equal, $M_{L_l}^{G_{nd}}(s_0,0)$ permutes without fixed point
the subspaces $\times_{i=1}^l \rho_{x(i)}$.
So the trace of the operator induced by $M_{L_l}^{G_{nd}}(s_0,0)$ on $V_\rho$ is zero.
Remain in the formula then only contributions from representations $\rho=\otimes_{i=1}^l \rho_i$ of $L_l$
such that all the $\rho_i$ are equal. So

$$I_{disc,t,\mu,G_{nd}}(f)=\sum_{\rho\in DS_{nd,t,\mu}}\tr(\rho(f))
+\sum_{l|nd,\, l\neq nd}\frac{1}{l^2}
 \sum_{\rho\in DS_{\frac{nd}{l},\frac{t}{l},\mu_l}}
\tr(M_{L_l}^{G_{nd}}(s_0,0)\, \rho^l(0,f)),$$ where
$DS_{k,\frac{t}{l},\mu_l}$ is the set of discrete series $\rho$ of
$G_k(\aa)$  such that $\rho$ is $\mu'$-equivariant for some
character $\mu'$ of $F^*_\infty$ such that $\mu'^l=\mu$ and the norm
of the imaginary part of its infinitesimal character is
$\frac{t}{l}$, and $\rho^l$ is the induced representation
$\rho\times\rho\times...\times\rho$ from $L_l$. In the last formula
we used the multiplicity one theorem for $G_k$, $k|nd$. The
representation $\rho$ being unitary, the representation $\rho^l$ is
irreducible and hence $M(s_0,0)$ act as a scalar on $\rho^l$. As it
is also a unitary operator, the scalar is some complex number
$\lambda_\rho$ of module 1.\\

The analogous definition $I_{disc,t,\mu,G'_{n}}(f')$ is given in [AC] for the groups $G'_n$ and
$f'\in H(G'_n)$. Then the authors show, equation (17.8) page 198, that, whenever $f\lra f'$, one has

\begin{equation}\label{ac}
I_{disc,t,\mu,G_{nd}}(f)=I_{disc,t,\mu,G'_{n}}(f').
\end{equation}

We have an easy lemma.

\begin{lemme}\label{nontransfer}
Let $l|nd$ and $\rho\in DS_\frac{nd}{l}$. Let $f'\in H(G'_n)$ and
$f\in H(G_{nd})$ such that $f\lra f'$. If $l$ does not divide $n$,
or if $\rho$ is not $D$-compatible, then $\tr(M(s_0,0)\,
\rho^l(f))=0$.
\end{lemme}

{\bf Proof.} Assume $l$ does not divide $n$. Then $d$ does not
divide $\frac{nd}{l}$. By class field theory the smallest common
multiple of the integers $d_v$ is $d$, so there exists a place $w$
such that $d_w$ does not divide $\frac{nd}{l}$. Then $\rho_w^l$ is
not $d_v$-compatible. The same, if $\rho$ is not $D$-compatible,
there exists a place $w$ such that $\rho_w$ is not $d_v$-compatible
and hence $\rho_w^l$ is not $d_v$-compatible.

In both cases we have then $\tr\rho_w^l(f_w)=0$
and as the operator $M(s_0,0)$ acts as a scalar, the result follows.
\qed

Another lemma:

\begin{lemme}\label{xxx}
Assume the multiplicity one theorem is true for all $G'_l$, $l<n$.
Then

(a)
$$I_{disc,t,\mu,G'_{n}}(f')=\sum_{\rho'\in DS'_{n,t,\mu}}m_{\rho'}\tr\rho'(f')+
\sum_{l|n,\, l\neq n}\frac{1}{l^2}
 \sum_{\rho'\in DS'_{\frac{n}{l},\frac{t}{l},\mu_l}}
\tr(M_{L'_l}^{G'_{n}}(s_0,0)\, \rho'^l(0,f')),$$
with the same notations as for $G_{nd}$ and where $m_{\rho'}$ is the multiplicity of $\rho'$ in the discrete spectrum.

(b) For all $f\lra f'$, one has
\begin{equation}\label{mainequa}
\sum_{\rho\in DS^D_{nd,t,\mu}}\tr\rho(f)+
\sum_{l|n,\, l\neq n}\frac{1}{l^2}\sum_{\rho\in DS^{D}_{\frac{n}{l},\frac{t}{l},\mu_l}}
\tr(M_{L_{l}}^{G_{nd}}(s_0,0)\, \rho^l(0,f))=
\end{equation}
$$
\sum_{\rho'\in DS'_{n,t,\mu}}m_{\rho'}\tr\rho'(f')+\sum_{l|n,\,
l\neq n} \frac{1}{l^2}\sum_{\rho'\in
DS'_{\frac{n}{l},\frac{t}{l},\mu_l}} \tr(M_{L'_l}^{G'_{n}}(s_0,0)\,
\rho'^l(0,f')),$$ where $DS^{D}_{?}$ is by definition the subset of
$D$-compatible representations in $DS_{?}$.
\end{lemme}

{\bf Proof.} (a) The proof is similar to the case $G_{nd}$.

(b) We used (a) and the equality \ref{ac}. But the $G_{nd}$ side has
been modified thanks to the lemma \ref{nontransfer}. The lemma
\ref{nontransfer} allows also the replacement of $DS_{?}$ by
$DS^{D}_{?}$.\qed
\ \\

Let us prove the theorem by induction on $n$. So we will use the
formula \ref{mainequa} among all. Let us point out, to not recall it
all the time, that the correspondence $\g$, once assumed or proven,
preserves the quantities $t$ and $\mu$.

First assume $n=1$. Then we get from the relation \ref{mainequa}:

\begin{equation}\label{cas1}
\sum_{\rho\in DS^D_{d,t,\mu}}\tr\rho(f)=\sum_{\rho'\in DS'_{1,t,\mu}}m_{\rho'}\tr\rho'(f').
\end{equation}
for all $f\lra f'$, where $m_{\rho'}$ is the multiplicity of $\rho'$ in the discrete spectrum.

\def\x{{\s'}}
Let us fix a representation $\x\in DS'_1$. Then we have $\x\in
DS'_{1,t,\mu}$ for some $t$ and $\mu$. We will show there exists
$\s\in  DS^D_{d,t,\mu}$ such that $|\lj|_v(\s_v)=\x_v$ for all $v\in
V$ and $\s_v=\x_v$ for all $v\notin V$, and also that $m_{\x}=1$.
Let $S$ be a finite set of places of $F$ containing all the places
in $V$, all the infinite places and all the places $v$ such that
$\x_v$ is not a spherical representation. For any $\pi\in
DS^D_{d,t,\mu}$ or $\pi\in DS'_{1,t,\mu}$ write $\pi_S$ for the
tensor product $\otimes_{v\in S}\pi_v$ and $\pi^S$ for the
restricted tensor product $\otimes_{v\notin S}\pi_v$. Let
$DS^D_{d,t,\mu,\x}$ (resp. $DS'_{1,t,\mu,\x}$) be the set of $\pi\in
DS^D_{d,t,\mu}$ (resp. $\pi\in DS'_{1,t,\mu}$) such that
$\pi^S=\x^S$. Then we have for all $f\lra f'$:
$$\sum_{\rho\in DS^D_{d,t,\mu,\x}}\tr\rho(f)=\sum_{\rho'\in DS'_{1,t,\mu,\x}}m_{\rho'}\tr\rho'(f').$$
This statement is inferred from the equation \ref{cas1} by a standard argument one may find
well expounded in [Fl2].
According to the strong multiplicity one theorem applied to $G_d$, the cardinality of
$DS^D_{d,t,\mu,\x}$ is either zero or 1. The cardinality of $DS'_{1,t,\mu,\x}$ is finite by [BB].
As $f_v=f'_v$ for $v\notin S$, we may simplify this equality with $\prod_{v\notin S}\tr \x_v(f'_v)$, choosing
$f'_v$ such that this product is not zero.
We get

$$\sum_{\rho\in DS^D_{d,t,\mu,\x}}\prod_{v\in S}\tr\rho_v(f_v)=
\sum_{\rho'\in DS'_{1,t,\mu,\x}}m_{\rho'}\prod_{v\in S}\tr\rho'_v(f'_v)$$
for functions such that $f_v\lra f'_v$ for all $v\in V$ and $f_v= f'_v$ for $v\in S\bc V$.
On the right side we have a finite
 non empty set (containing at least
$\x$) of distinct characters on a finite product of groups.
The linear independence of characters on these groups implies the linear independence of characters
on the product, and so there exist functions $f'_v\in H'(G'_{1,v})$ for $v\in S$, supported
on the set of regular semisimple elements, such that the right side of the equality does not vanish on
$(f'_v)_{v\in S}$. Then $DS^D_{d,t,\mu,\x}$ is not empty and hence contains one element. Let us call this element $\s$.
As $\s$ is $D$-compatible, for every $v\in V$ we have that
$|\lj|_v(\s_v)$ is an irreducible unitary representation $u'_v$ of $G'_{1,v}$ such that
$\tr(\s_v(f_v))=\tr(u'_v(f'_v))$ for all $f_v\lra f'_v$. So by linear independence of characters on the group
$\times_{v\in S}G'_{1,v}$ we have to have $u'_v=\x_v$ for all $v\in V$ and $\s_v=\s'_v$ for all
$v\in S\bc V$.
This obviously implies $m_\x=1$ which is (b).
Now $\g(\x)$ is defined. To show the surjectivity of $\g$ onto
 $DS^D_{d}$ one starts with $\s\in DS^D_{d,t,\mu}$ and set $S$ to
be a finite set of places containing all the places in $V$, all the
infinite places and all the places $v$ such that $\s_v$ is not
spherical. Then the proof is the same. Now we have $\x_v=\g(\x)_v$
for all $v\notin V$. The strong multiplicity one theorem for $G'_d$
implies then both the fact that the map $\g$ is injective
(completing the proof of (a)) and the strong multiplicity one
theorem for $G'_1$ ((c)). The point (d) is obtain now by transfer
under $\g^{-1}$ and the proposition \ref{stable} (b).

We finished the proof of the theorem for $n=1$.\\

Let us now assume the theorem has been proven for all $k<n$ and call $\g_k$ the transfer map at level $k$.
This hypothesis enables us to apply the lemma \ref{xxx} and to quote the relation (\ref{mainequa}):

\begin{equation}
\sum_{\rho\in DS^D_{nd,t,\mu}}\tr\rho(f)+
\sum_{l|n,\, l\neq n}\frac{1}{l^2}\sum_{\rho\in DS^{D}_{\frac{n}{l},\frac{t}{l},\mu_l}}
\tr(M_{L_{l}}^{G_{nd}}(s_0,0)\, \rho^l(0,f))=
\end{equation}
$$
\sum_{\rho'\in DS'_{n,t,\mu}}m_{\rho'}\tr\rho'(f')+\sum_{l|n,\, l\neq n}
\frac{1}{l^2}\sum_{\rho'\in DS'_{\frac{n}{l},\frac{t}{l},\mu_l}}
\tr(M_{L'_l}^{G'_{n}}(s_0,0)\, \rho'^l(0,f')).$$

Moreover, using the point (d) of the theorem for ${\bf G}_k$, $k<n$,
the induction hypothesis implies that the representations $\rho'^l$
are irreducible (proposition  \ref{unit} (b)). So
$M_{L'_l}^{G'_{n}}(s_0,0)$ is again a scalar and as it is unitary
the scalar is a complex number $\lambda_{\rho'}$ of module 1. So the
equation is actually, using again the induction to transfer the
representations in $DS^{D}_{\frac{n}{l},\frac{t}{l},\mu_l}$:

\begin{equation}
\sum_{\rho\in DS^D_{nd,t,\mu}}\tr\rho(f)+
\sum_{l|n,\, l\neq n}\frac{1}{l^2}\sum_{\rho\in DS^{D}_{\frac{n}{l},\frac{t}{l},\mu_l}}
\lambda_{\rho}\tr(\rho^l(0,f))=
\end{equation}
$$
\sum_{\rho'\in DS'_{n,t,\mu}}m_{\rho'}\tr\rho'(f')+\sum_{l|n,\,
l\neq n} \frac{1}{l^2}\sum_{\rho\in
DS^D_{\frac{n}{l},\frac{t}{l},\mu_l}}
\lambda_{\g^{-1}_{\frac{n}{l}}(\rho)}\tr(\rho^l(0,f))$$ for $f\lra
f'$.

Now the proof goes as for the case $n=1$ with a minor modification
in the end. Chose a representation $\x\in DS'_{n,t,\mu}$. Fix a
finite set $S$  of places of $F$ which contains all the places in
$V$, all the infinite places and all the places $v$ for which $\x_v$
is not spherical. By the theorem of multiplicity one for $G_{nd}$
the set $A$ of $\s\in DS^D_{n,t,\mu}$ such that $\s^V=\x^V$ is empty
or contains only one element. If we apply the proposition
\ref{multone} to the representations $\rho^l$ and all the places out
of $S$, then we conclude that the set $B$ of representation
$\gamma=\rho^l$ (where $l|n$ and $l<n$) such that $\gamma^V=\x^V$ is
empty or contains one element.  Let $DS'_{n,t,\mu,\x}$ be the set of
$\tau'\in DS'_{n,t,\mu}$ such that $\tau'^S=\s'^S$. Then
$DS'_{n,t,\mu,\x}$ is not empty (contains $\s'$) and finite ([Ba3];
we do not quote [BB] again since the representations may not be
cuspidal).

By the same argument in [Fl2] that we quoted for the case $n=1$ we
obtain then

$$\sum_{\s\in A}\prod_{v\in S}\tr\s_v(f_v)+
\sum_{\gamma\in
B}\frac{\lambda_\gamma-\lambda_{\g^{-1}_{\frac{n}{l}}(\gamma)}}{l^2}
\prod_{v\in S}\tr\gamma_v(f_v)= \sum_{\rho'\in
DS'_{n,t,\mu,\x}}m_{\rho'}\prod_{v\in S}\tr\rho'_v(f'_v)$$ if
$f_v\lra f'_v$ for all $v\in V$ and $f_v\lra f'_v$ for all $v\in
S\bc V$.

If $A$ is not empty and $\s$ is the unique element of $A$, then the
local components of $\s$ are unitary and we can transfer them. If
$B$ is not empty and $\gamma$ is the unique element of $B$, then the
local components of $\gamma$ are unitary and we can transfer them.
In any possible case we do so. But {\it the coefficient
$\frac{\lambda_\gamma-\lambda_{\g^{-1}_{\frac{n}{l}}(\gamma)}}{l^2}$
cannot be an integer because its module is less than $\frac{1}{2}$}.
So the linear independence of characters on the group $\times_{v\in
S}G'_{v}$ implies that $B$ was empty, $A$ was not empty, on the
right side there is only $\s$ and $m_\s=1$. The rest of the proof is
exactly like for $n=1$.\qed

\begin{cor}
The intertwining operators $M_{L_l}^{G_{nd}}(s_0,0)$ and
$M_{L'_l}^{G'_{n}}(s_0,0)$ are given by the same scalar. In
particular, the computations in [KS] transfer to $G'_n(\aa)$.
\end{cor}

{\bf Proof.} This is the consequence of
$\lambda_\gamma-\lambda_{\g^{-1}_{\frac{n}{l}}(\gamma)}=0$ implied
by the end of the proof of the theorem.\qed

\subsection{A classification for discrete series and automorphic representations of $G'_n$}

If $L=\times_{i=1}^kG'_{n_i}$ is a standard Levi subgroup of $G'_n$,
we call {\bf essentially square integrable} (resp. {\bf cuspidal})
representation of $L$ a representation
$\pi'=\otimes_{i=1}^k\nu^{a_i}\rho'_i$ where, for each $i$,
$\rho'_i$ is a discrete series  (resp. cuspidal representation) of
$G'_{n_i}$ and $a_i$ is a real number. The representation $\pi'$ is
the said to be $D$-{\bf compatible} if all the $\rho_i$ are
$D$-compatible.

\begin{prop}\label{cuspidal}
Let $\rho\in DS_m$ be a cuspidal representation. Let $s_{\rho,D}$ be
the smallest common multiple of $s_{\rho_v,d_v}$, $v\in V$ (cf. section  \ref{generic}).
Then

(a) $MW(\rho,k)$ is $D$-compatible if and only if $s_{\rho,D}|k$.

(b) ${\g}^{-1}(MW(\rho,s_{\rho,D}))=\rho'\in DS'_{\frac{ms_{\rho,D}}{d}}$
  is cuspidal (in particular $\g^{-1}$ sends cuspidal to cuspidal).
\end{prop}

{\bf Proof.} (a) This is an easy consequence of the discussion in section \ref{generic} and the definition
of $s_{\rho,D}$.

(b) Assume $\rho'$ is not cuspidal. Then there exists an essentially
cuspidal representation $\tau'$ of a proper standard Levi subgroup
$L'$ of $G'_n$ such that $\pi'$ is a constituent of the induced
representation to $G'_{\frac{ms_{\rho,D}}{d}}$ from $\tau'$. Set
$\tau=\g(\tau')$. So $\tau$ is a $D$-compatible essentially square
integrable representation of $L(\aa)$ where $L$ is a proper standard
Levi subgroup of $G_{ms_{\rho,D}}$ corresponding to $L'$. By the
theorem 4.4 of [JS], $\tau$ has the same cuspidal support as
$MW(\rho,s_{\rho,D})$. As it is a $D$-compatible essentially square
integrable representation and lives on a smaller subgroup, this
contradicts the minimality of $s_{\rho,D}$.\qed

\begin{rem}
It will be proved in the Appendix that all the cuspidal
representations of $G'_n(\aa)$ are obtained like in the proposition
\ref{cuspidal}. But at this point this proof cannot be made yet, so
for now we will call these representations {\rm basic cuspidal}.
After, using the next proposition, Grbac will prove in the Appendix
that basic cuspidal and cuspidal is the same thing, the reader may
drop the word "basic" in the next proposition and have a clean
classification.
\end{rem}

Let us call {\bf basic cuspidal} a cuspidal representation obtained
like the $\rho'={\g^{-1}}(MW(\rho,s_{\rho,D}))$ in the point (b) of
the proposition . We then set $s(\rho')=s_{\rho,D}$ and
$\nu_{\rho'}=\nu^{s_{\rho,D}}$. If $L=\times_{i=1}^kG'_{n_i}$ is a
standard Levi subgroup of $G'_n$, we call {\bf basic essentially
cuspidal} representation of $L$ a representation
$\otimes_{i=1}^k\nu^{a_i}\rho'_i$ where, for each $i$, $\rho'_i$ is
a basic cuspidal representation of $G'_{n_i}$ and $a_i$ is a real
number.

We now give a classification of discrete series of groups $G'_n$.
The point (a) generalizes [MW2] and the point (b) generalizes the
thorem 4.4 in [JS].

\begin{prop}\label{classif}
(a) Let $\rho'\in DS'_m$ be a basic cuspidal representation. Let
$k\in \n^*$. The induced representation
$\prod_{i=0}^{k-1}(\nu_{\rho'}^{\frac{k-1}{2}-i}\rho')$ has a unique
constituent $\pi'$ which is a discrete series. We write then
$\pi'=MW'(\rho',k)$. Every discrete series $\pi'$ of a group $G'_n$,
$n\in\n^*$, is of this type, and $k$ and $\rho'$ are determined by
$\pi'$. The discrete series $\pi'$ is basic cuspidal if and only if
$k=1$. If $\pi'=MW'(\rho',k)$, then
${\g}(\rho')=MW(\rho,s_{\rho,D})$ if and only if
${\g}(\pi')=MW(\rho,ks_{\rho,D})$.

(b) Let $(L_i,\rho'_i)$, $i=1,2$, be such that $L_i$ is a standard
Levi subgroup of $G'_n$ and $\rho'_i$ is a basic essentially
cuspidal representation of $L_i(\aa)$ for $i=1,2$. Fix any finite
set of places $V'$ containing the infinite places and all the finite
places where $\rho'_1$ or $\rho'_2$ is not spherical. If, for all
places $v\notin V'$, the spherical subquotients of the induced
representations from $\rho'_{i,v}$ to $G'_n$ are equal, then the
couples $(L_i,\rho'_i)$ are conjugated.

(c) If $\pi'$ is an automorphic representation of $G'_n$, then there
exists a couple $(L,\rho')$ where $L$ is a  standard Levi subgroup
of $G'_n$ and $\rho'$ is a basic essentially cuspidal representation
of $L(\aa)$ such that $\pi'$ is a constituent of the induced
representation from $\rho'$ to $G'_n(\aa)$. The couple $(L,\rho')$
is unique up to conjugation.
\end{prop}

{\bf Proof.} (a) Let $\g(\rho')=MW(\rho,s_{\rho,D})$. The discrete series $MW(\rho,ks_{\rho,D})$ is
$D$-compatible (proposition  \ref{cuspidal} (a)).
We will show directly that $\g^{-1}(MW(\rho,ks_{\rho,D})$
is a constituent of $\prod_{i=0}^{k-1}(\nu_{\rho'}^{\frac{k-1}{2}-i}\rho')$.

It is enough to show that, for every place $v\in V$,
$|\lj|_v(MW(\rho,ks_{\rho,D})_v)$ is a subquotient of the local
representation
$\prod_{i=0}^{k-1}(\nu_{\rho'}^{\frac{k-1}{2}-i}\rho'_v)$. By
proposition  \ref{reduc}, it is enough to show that the esi-support
of $|\lj|_v(MW(\rho,ks_{\rho,D})_v)$ is the reunion of the
esi-supports of representations
$\nu_{\rho'}^{\frac{k-1}{2}-i}\rho'_v$. As in the section
\ref{generic}, we may write the generic representation $\rho_v$ as a
product of essentially square integrable representations
$\prod_{j=1}^m\nu^{e_j}\s_j$ and we have seen then that
$$\rho'_v=|\lj|_v(Lg(\rho_v,s_{\rho,D}))=\prod_{j=1}^m\nu^{e_j}|\lj|_v(u(\s_j,s_{\rho,D}))$$
and
$$|\lj|_v(Lg(\rho_v,ks_{\rho,D}))=\prod_{j=1}^m\nu^{e_j}|\lj|_v(u(\s_j,ks_{\rho,D})).$$
Fix an index $j$. If $\s_j$ transfers to $\s'_j$ (case (a) of the proposition  \ref{main}),
we know that $|\lj|_v(u(\s_j,s_{\rho,D}))={\bar u}(\s'_j,s_{\rho,D}))$ and
$|\lj|_v(u(\s_j,ks_{\rho,D}))={\bar u}(\s'_j,ks_{\rho,D})$. One may easily verify that the esi-support of
${\bar u}(\s'_j,ks_{\rho,D})$ is the reunion of the esi-supports of
$\nu^{(\frac{k-1}{2}-i)s_{\rho,D}}{\bar u}(\s'_j,s_{\rho,D})$
for $i\in\{1,...,k\}$. If $\s_j$ does not transfer (case (b) of the proposition  \ref{main}),
one has to use the formula \ref{product} in the section \ref{generic} involving $\s'_{j+}$ and $\s'_{j-}$,
but then the proof
goes exactly the same as for the case when $\s_j$ transfers.

So $\prod_{i=0}^{k-1}(\nu_{\rho'}^{\frac{k-1}{2}-i}\rho')$ has a constituent $\pi'$ which is a discrete series.
The strong
multiplicity one theorem for discrete series
of $G'_n$ (proposition  \ref{correspondence} (c)) implies this induced representation has no other constituent
which is a discrete series.

Let $\pi'\in DS'_n$ be a  discrete series and let us show it is
obtained like this. Set $\g(\pi')=MW(\rho,p)$. We have
$s_{\rho,D}|p$ since $MW(\rho,p)$ is $D$-compatible (proposition
\ref{cuspidal} (a)). So, if we set
$\rho'=\f^{-1}(MW(\rho,s_{\rho,D}))$, $\rho'$ is a basic cuspidal,
and we have $\pi'=MW'(\rho',\frac{p}{s_{\rho,D}})$. The strong
multiplicity one theorem for $G_{nd}$ implies $p$ and $\rho$ are
determined by $\pi'$, so $k=\frac{p}{s_{\rho,D}}$ and $\rho'$ are
determined by $\pi'$. It is clear that $\pi'$ is basic cuspidal if
and only if $p=s_{\rho,D}$, if
and only if $k=1$.\\

(b) $\g(\rho'_1)=\rho_1$ is a tensor product of the form $\otimes_{i=1}^{p_1}\nu^{\a_i}MW(\xi_i,s_{\xi_i,D})$
and $\g(\rho'_2)=\rho_2$ is a tensor product of the form $\otimes_{j=1}^{p_2}\nu^{\beta_j}MW(\tau_j,s_{\tau_j,D})$,
where $\xi_i$ and $\tau_j$ are cuspidal. As the induced representations to $G_{nd}$
from $\rho_1$ and $\rho_2$ has equal spherical subquotient
at all finite places which are not in $V\cup V'$,
we know that the essentially cuspidal supports of $\rho_1$ and $\rho_2$ are equal (theorem 4.4 in [JS]).
As $\xi_i$ and $\tau_j$ are cuspidal,
it follows from the formulas of $\rho_1$ and $\rho_2$ that the multisets
$\{(\a_i,\xi_i)\}$ and $\{(\beta_j,\tau_j)\}$ are equal and so the tensor products are the same up to permutation.\\

(c) The existence is proven in (a). The unicity in (b).\qed

\subsection{Further comments}

The question whether the transfer of discrete series could be extended to unitary automorphic representations or
not seems natural. Let us extend in an obvious way
the notion of $D$-compatible from discrete series to
 unitary automorphic
representations of $G_{nd}(\aa)$. Let us formulate two questions.\\
\ \\
{\bf Question 1.} Given a unitary automorphic representation $a'$ of
$G'_n(\aa)$, is it possible to find a unitary automorphic
representation $a$ of $G_{nd}(\aa)$ such that $a_v=a'_v$ for all
$v\notin V$ and
$|\lj|_v(a_v)=a'_v$ all $v\in V$?\\
\ \\
{\bf Question 2.} Given a $D$-compatible unitary automorphic
representation $a$ of $G_{nd}(\aa)$, is it possible to find a
unitary automorphic representation $a'$ of $G'_{n}(\aa)$ such that
$a_v=a'_v$ for all $v\notin V$ and $|\lj|_v(a_v)=a'_v$ all $v\in
V$?\\

These questions are independent and the answer is in general ``no'' for both.\\

Consider the first question. Roughly speaking the counterexample
comes from the fact that there exist unitary irreducible
representations of an inner form of $GL_n$ over a local field which
do not correspond to a unitary representation of $GL_n$. The problem
is to realize such a representation as a local component of a
unitary automorphic representation. Here is the construction, based
on the lemma \ref{nonunit}.

Let $dim_FD=16$. Let $G'=GL_3(D)$. Assume there is a finite place
$v_0$ of $F$ such that the local component of $G'(\aa)$ at the place
$v_0$ is $G'_{v_0}\simeq GL_{3}(D_{v_0})$ with
$dim_{F_{v_0}}D_{v_0}=16$. It is possible to chose such a $D$ by
global class field theory. Let $\rho'$ be a cuspidal representation
of $G'(\aa)$ such that $\rho'_{v_0}$ is the Steinberg representation
of $G'_{v_0}$. Then $\g(\rho')$ is cuspidal. Indeed, its local
component at the place $v_0$ has to be the Steinberg representation
of $GL_{12}(F_{v_0})$ (the only unitary irreducible elliptic
representations being the trivial representation and the Steinberg
representation). In particular $s_{\rho'}=1$.

Let $\tau'=MW'(\rho',16)$. Let $St'_3$ be the Steinberg
representation of $GL_3(D_{v_0})$ and $St'_4$ the Steinberg
representation of $GL_4(D_{v_0})$. Then $\tau'_{v_0}=u'(St'_3,16)$.

Let $\tau''$ be the global representation defined by:
$\tau''_v=\tau'_v$ for all $v\neq v_0$ and
$\tau''_{v_0}=\nu^{-\frac{3}{2}}u'(St'_3,4)\times\nu^{-\frac{1}{2}}
u'(St'_4,3)\times\nu^{\frac{1}{2}} u'(St'_4,3)\times
\nu^{\frac{3}{2}}u'(St'_3,4).$ Let us show that $\tau''$ is an
automorphic representation. We have $\tau''_{v_0}<\tau'_{v_0}$ by
the lemma \ref{nonunit} (ii). So $\tau''_{v_0}$ is a subquotient of
$\times_{i=1}^{16} \nu^{\frac{17}{2}-i}St'_3$. So $\tau''$ is a
constituent of $\times_{i=1}^{16} \nu^{\frac{17}{2}-i}\rho'$. As
$\rho'$ is cuspidal, $\tau''$ is automorphic. All the local
components of $\tau''$ are unitary. It is true by definition for
$\tau''_v$, $v\neq v_0$, and by lemma \ref{nonunit} (i) for
$\tau''_{v_0}$. So $\tau''$ is a unitary automorphic representation.
It cannot correspond to a unitary automorphic representation of
$GL_{48}(\aa)$
because by lemma \ref{nonunit} (iii) there is a transfer problem at the place $v_0$.\\

Consider now the second question.  Let $dim_FD=d^2=4$. Let
$G'=GL_3(D)$. Assume there is a finite place $v_0$ of $F$ such that
the local component of $G'(\aa)$ at the place $v_0$ is
$G'_{v_0}\simeq GL_{3}(D_{v_0})$ with $dim_{F_{v_0}}D_{v_0}=4$. For
all $i\in \n^*$, write $St_i$ for the Steinberg representation of
$GL_i(F_{v_0})$ and $St'_i$ for the Steinberg representation of
$GL_i(D_{v_0})$. Let $\rho$ be a cuspidal representation of
$GL_3(\aa)$ such that $\rho_{v_0}=St_3$. Set $\tau=MW(\rho,2)$. We
have $s_{\rho,D}=2$ (since $s_{\rho,D}$ always divides $d$ and here
$d=2$ and $s_{\rho,D}\neq 1$). So $\tau$ is $D$-compatible and
$\tau'=\g^{-1}(\tau)$ is a cuspidal representation. We have
$\tau_{v_0}=u(St_3,2)$. Let $\pi$ be the representation $St_4\times
St_2$ of $GL_6(F_{v_0})$. Then $\pi$ is tempered. We also have
$\pi<\tau_{v_0}$, so $\pi$ is a subquotient of
$\nu^{\frac{1}{2}}St_3\times\nu^{-\frac{1}{2}}St_3$. So the
representation $\xi$ defined by $\xi_v=\tau_v$ if $v\neq v_0$ and
$\xi_{v_0}=\pi$ is a constituent of
$\nu^{\frac{1}{2}}\rho\times\nu^{-\frac{1}{2}}\rho$, hence an
automorphic representation. All its local components are unitary. It
is a $D$-compatible representation because $\pi$ is $2$-compatible.
Let us show that the representation $\xi'$ defined by
$\xi'_v=|\lj|_v(\xi_v)$ for all places $v$ of $F$ is not
automorphic. For every place $v\neq v_0$, we have $\xi'_v=\tau'_v$.
As $\tau'$ is cuspidal, it is enough to show that $\xi'\neq \tau'$
by the theorem \ref{classif} (b) applied to $\tau'$ and the cuspidal
support of $\xi'$. So this comes to show that
$|\lj_{v_0}|(u(St_3,2))\neq |\lj_{v_0}|(\pi)$. Using the formulas we
have for the transfer (proposition  \ref{transfer}) we find
$|\lj_{v_0}|(u(St_3,2))= u(St'_1,3)$ and
$|\lj_{v_0}|(\pi)=St'_2\times St'_1$. If $1_2$ is the trivial
representation of $GL_2(D_{v_0})$, we have $u(St'_1,3)=1_2\times
St'_1$ hence $\xi'_{v_0}\neq \tau'_{v_0}$.

\def\s{{\sigma}}
\def\stf{St_{n}}
\def\std{St'_{n}}
\def\unf{1_{n}}
\def\und{1'_{n}}
\def\nuf{\nu_{n}}
\def\nud{\nu_{n}'}
\def\ccc{{\bf C}}
\def\d{{\mathcal D}}

\section{$L$-functions and $\epsilon'$-factors}

In this section we examine the transfer of $L$-functions and
$\epsilon'$-factors. Nothing is original, the results
are simple computations using [GJ] and [Ja].\\

Let $F$ be a non-Archimedean local field of any characteristic and
$D$ a division algebra of dimension $d^2$ over $F$. For all $n$,
recall that $G_n=GL_n(F)$ and $G'_n=GL_n(D)$.

Suppose the characteristic of the residual field of $F$ is $p$ and its cardinal is $q$.
Let $O_F$ be the ring of integers of $F$ and $\pi_F$ be a uniformizer of $F$.
Fix a additive character $\psi$ of $F$ trivial on $O$ and non trivial on $\pi_F^{-1}O$.
For irreducible representations $\pi$ of $G_n$ or $G'_n$, we adopt the notations  $L(s,\pi)$ and
$\epsilon'(s,\pi,\psi)$ for the $L$-function and the $\epsilon'$-factor, as defined in [GJ].

In this section we will specify $\nu$, because confusion may appear.
For all $n\in \n^*$, $\nuf$ (resp. $\nud$) will denote the absolute
value of the determinant on $G_n$ (resp. $G'_n$); $\unf$ (resp.
$\und$) will denote the trivial representation of $G_n$ (resp.
$G'_n$); let $\stf=Z^u(1_1,n)$ (resp. $\std=T^u(1'_1,n)$) be the
{\bf Steinberg representation} of $G_n$ (resp. $G'_n$). One has
$\stf=|i(\unf)|$ and $\std=|i'(\und)|$. The character of the
Steinberg representation is constant on the set of elliptic
elements, equal to $(-1)^{n-1}$. In particular, we have
$\ccc(St_{d})=1'_{1}$. This implies that $s(1'_{1})=d$ (here
$s(1'_{1})$ is the invariant defined at the section \ref{esi2},
nothing to do with the complex variable $s$). For all $n\in\n^*$,
one has $\ccc(St_{nd})=St'_{n}$.

We bring together facts from [GJ] in the following theorem:

\begin{theo}\label{lfct}
a) We have $L(s,1'_{1})= (1-q^{-s-\frac{d-1}{2}})^{-1}$,
$$L(s,\und)=\prod_{j=0}^{n-1}L(s+d\frac{n-1}{2}-dj,1'_{1})=\prod_{j=0}^{n-1}(1-q^{-s+dj-\frac{dn-1}{2}})^{-1}$$
and
$$\epsilon'(s,\und,\psi)=\prod_{j=0}^{n-1}\epsilon'(s+d\frac{n-1}{2}-dj,1'_{1},\psi)=
\prod_{j=0}^{dn-1}\epsilon'(s+\frac{dn-1}{2}-j,1_{1},\psi).$$

(b) We have
$L(\std)=L(s+d\frac{n-1}{2},1'_{1})=(1-q^{-s-\frac{dn-1}{2}})^{-1}$
and
$$\epsilon'(s,\std,\psi)=\prod_{j=0}^{n-1}\epsilon'(s+d\frac{n-1}{2}-dj,1'_{1},\psi)=
\prod_{j=0}^{dn-1}\epsilon'(s+\frac{dn-1}{2}-j,1_{1},\psi).$$

(c) If $\rho'$ is a cuspidal representation of $G'_x$, then
$L(s,\rho')=1$ unless $x=1$ and $\rho'$ is an unramified character
of $D^\times$. If $x=1$ and $\rho'$ is an unramified character of
$D^\times$, then $\rho'=\nu_{1}'^t$ for some $t\in \cc$ and we have
$L(s,\rho')=(1-q^{-s-t-\frac{d-1}{2}})^{-1}$.

(d) Let $\sigma'=T(\rho',k)$ be an essentially square integrable representation of $G'_{xk}$ where $\rho'$
is a cuspidal representation of
$G'_x$. Then $L(s,\sigma')=L(s,\rho')$.

In particular,  $L(s,\sigma')=1$ unless $x=1$ and $\rho'$ is an
unramified character of $D^\times$. If $x=1$ and $\rho'$ is an
unramified character of $D^\times$ then $\rho'=\nu_{1}'^t$ for some
$t\in \cc$ and then $\sigma'=\nud^{t+d\frac{n-1}{2}} \std$. We have
$L(s,\sigma')=(1-q^{-s-t-\frac{d-1}{2}})^{-1}$ in this case.

We have, in general,
$$\epsilon'(s,\sigma',\psi)=\prod_{j=0}^{k-1}\epsilon'(s+js(\sigma'),\rho',\psi)$$
(in this formula, $s(\s')$ is the invariant defined at section \ref{esi2}).

(e) Let $\s'_i\in \d^{'u}_{n_i}$, $i\in\{1,2,...,k\}$,
$\sum_{i=1}^k=n$. Let $a_1\geq a_2\geq ...\geq a_k$ be real numbers.
Set $S'=\times_{i=1}^k\nu_{n_i}'^{a_i}\s'_i$ and $\pi'=Lg(S')$.

Then
$$L(s,\pi')=\prod_{i=1}^kL(s,\sigma'_i)$$
and
$$\epsilon'(s,\pi',\psi)=\prod_{i=1}^k\epsilon'(s,\sigma'_i, \psi).$$

In particular, if $\rho'_1,\rho'_2,...,\rho'_p$ is the cuspidal support of $\pi'$, then
$$\epsilon'(s,\pi',\psi)=\prod_{i=1}^p\epsilon'(s,\rho'_i, \psi).$$
\end{theo}

{\bf Proof.} (a) This is shown in the proposition  6.11 in [GJ], where the formula is slightly wrong.
The reader may verify that the good formula for the $L$-function in [GJ], proposition  6.9 is with $(d-1)$
instead of $(n-1)$, as indicated by the proof of this proposition . Then this typo error is propagated to [GJ],
proposition  6.9, where the reader may easily verify that the right formula obtained, once corrected the proposition  6.9,
is our formula. For the $\epsilon'$-factor our formula fits the [GJ] one.

(b) The $\epsilon'$-factor of $\std$ equals the $\epsilon'$-factor
of $\und$ as they are both sub-quotients of the same induced
representation ([GJ], corollary  3.6).

Let us check the $L$-function. For the particular case $D=F$, the computation of the $L$-function is theorem 7.11 (4),
[GJ]. Let us give a general (different) proof by induction on $n$.

For $n=1$ we have $\std=St'_{1}=1'_{1}$ and the result is implied by
(a).

For any $n>1$, the representation $\std$ is a subquotient of the
induced representation from
$\nu_{1}'^{-\frac{d(n-1)}{2}}1'_{1}\otimes
\nu_{n-1}'^{\frac{d}{2}}St'_{n-1}$. We know that
$$L(\nu_{1}'^{\frac{d(n-1)}{2}}1'_{1})=(1-q^{-s-\frac{d-1}{2}+\frac{d(n-1)}{2}})^{-1}$$
and, by the induction assumption, we have
$$L(s,\nu_{n-1}'^{\frac{d}{2}}St'_{n-1})=(1-q^{-s-\frac{dn-1}{2}})^{-1}.$$

By [GJ], corollary  3.6, $L(s,\std)$ is equal to one of these two
functions or to their product. But, by [GJ], proposition  1.3 and
theorem 3.3 (1) and (2), the poles of $L(s,\std)$ cannot be greater
than $\frac{d(n-1)}{2}-\frac{dn-1}{2}=-\frac{d-1}{2}$, so there is
no positive pole (this trick comes from the original proof: an
$L$-function of a square integrable representation cannot have a
pole with  positive real part). So $L(s,\std)=
L(s,\nu_{n-1}'^{\frac{d}{2}}St'_{n-1})=(1-q^{-s-\frac{dn-1}{2}})^{-1}.$

(c) The first assertion is a consequence of lemma 4.1, proposition  4.4 and proposition  5.11 of [GJ]
(prop 5.11 is not enough, since the authors assume $m>1$ at the beginning of the section 5).
The second assertion is a direct consequence of the point (a) of the present theorem.

(d) For the particular case of $G_n$ this is explained below
proposition  3.1.3 of [Ja]. The same proof apply to $G'_n$, using
the calculus for $St'_{1}$, i.e. the point (b).

(e) This is proven in [Ja] for $G_n$, but the same proof apply to $G'_n$.\qed

\begin{theo}
Let $\ccc$ be the local Jacquet-Langlands correspondence between $G_{nd}$ and $G'_n$.
Then, for all
$\sigma\in \d^u_n$,
we have $L(s,\s)=L(s,\ccc(\s))$ and
$\epsilon' (s,\s,\psi)=\epsilon' (s,\ccc(\s),\psi)$.
\end{theo}

{\bf Proof.} Let us show it first for the Steinberg representation
and its twists. We have $\ccc(St_{nd})=\std$. The theorem \ref{lfct}
(a) and (b) implies the statement in this case. This implies then
the statement for all the twist of $St_{nd}$ with characters.

\begin{lemme}
For all
$\sigma\in \d^u_{nd}$,
we have $\epsilon' (s,\s,\psi)=\epsilon' (s,\ccc(\s),\psi)$.
\end{lemme}

{\bf Proof.} The proof is standard, using an easy global
correspondence
(true in all characteristics) and the previous
calculus for the Steinberg representations.
See for example [Ba2], page 741 : {\it Les facteurs $\epsilon'$}.\qed
\ \\

Let us complete the proof of the theorem with the calculus of $L$-functions.
If $\sigma\in \d^u_{nd}$ or $\d^{'u}_n$
which is not a twist of the Steinberg representation,  then by theorem \ref{lfct} d) implies that its
$L$-function is trivial and so its $\epsilon'$-factor is equal to its $\epsilon$-factor.
As $\ccc(\s)$ is a twist of the Steinberg representation
if and only if $\s$ itself is a twist of the Steinberg representation, the statement has been now proven for all
$\s\in \d^u_{nd}$.\qed

\begin{cor}\label{epsilon}
Let $\s'_i\in \d^{'u}_{n_i}$, $i\in\{1,2,...,k\}$, $\sum_{i=1}^k=n$.
Let $a_1\geq a_2\geq ...\geq a_k$ be real numbers. Set
$S'=\times_{i=1}^k\nu_{n_i}'^{a_i}\s'_i$. Let
$\ccc^{-1}(\s'_i)=\s_i\in\d^u_{dn_i}$ and set
$S=\times_{i=1}^k\nu_{n_id}^{a_i}\s_i$. Then
$L(s,Lg(S'),\psi)=L(s,Lg(S),\psi)$ and
$\epsilon'(s,Lg(S'),\psi)=\epsilon'(s,Lg(S),\psi)$.
\end{cor}

{\bf Proof.} This is implied by the previous theorem and the point
(e) of the theorem \ref{lfct}.\qed

\begin{cor}
Assume the characteristic of $F$ is zero.
If $u\in Irr^u_{nd}$ is such that ${\bf LJ}_n(u)\neq 0$. Then $\epsilon'(s,u,\psi)=\epsilon'(s,|{\bf LJ}|_n(u),\psi)$.
\end{cor}

{\bf Proof.} It is enough to prove it for $u=u(\s,k)$, $\s\in \d^u_{p}$, $k,p\in\n^*$,
such that $|{\bf LJ}_{pk}|(u)=u'\neq 0$. If we are in the case (a) of the
proposition  \ref{main}, then  $u$ and $u'$ are like in the
corollary  \ref{epsilon}. In particular, their $L$ functions are equal too.
If we are in the case (b) of the proposition  \ref{main}, then
$|i(u)|$ and $|i'(u')|$ are like in the corollary  \ref{epsilon}.
Now, the $\epsilon'$-factor depends only on
the cuspidal support (theorem \ref{lfct} e)). So the $\epsilon'$-factor is the same for an irreducible representation
and its dual. But in general we do not get equality for the $L$-functions in this case.\qed

\section{Bibliography}

[AC] J.Arthur, L.Clozel, {\it Simple Algebras, Base Change, and the
Advanced Theory of the Trace Formula},  Ann. of Math. Studies,
Princeton Univ. Press 120, (1989).

[Au] A.-M. Aubert, Dualit\'e dans le groupe de Grothendieck
de la cat\'egorie des repr\'esentations lisses de longeur
finie d'un groupe r\'eductif $p$-adique, {\it Trans. Amer. Math. Soc.},
347 (1995), 2179-2189 (there is an Erratum: {\it Trans. Amer. Math. Soc.},
348 (1996), 4687-4690).

[Ba1] A.I.Badulescu, Orthogonalit\'e des caract\`eres pour $GL_n$ sur un corps local de caract\'eristique non nulle,
{\it Manuscripta Math.}, 101 (2000), 49-70.

[Ba2] A.I.Badulescu, Correspondance de Jacquet-Langlands en caract\'eristique
non nulle, {\it Ann. Scient. \'Ec. Norm. Sup.} 35 (2002), 695-747.

[Ba3] A.I.Badulescu, Un r\'esultat de finitude dans le spectre
automorphe pour les formes int\'erieures de $GL_n$ sur un corps global,
{\it Bull. London Math. Soc.} 37 (2005), 651-657.

[Ba4] A.I.Badulescu, Jacquet-Langlands et unitarisabilit\'e, to appear
in {\it Journal de l'Institut de Math\'ematiques de Jussieu}, may be
found at http://www-math.univ-poitiers.fr/\~\ badulesc/articole.html.

[BB] A.I.Badulescu, Un th\'eor\`eme de finitude {\it Compositio
Mathematica} 132 (2002), 177-190, with an Appendix by P.Broussous.

[Be] J.N.Bernstein, $P$-invariant distributions on $GL(N)$ and the classification of unitary representations of $GL(N)$
(non-Archimedean case), in
{\it Lie groups and representations II}, Lecture Notes in Mathematics 1041, Springer-Verlag, 1983.

[BJ] Automorphic forms and automorphic representations,
{\it Automorphic forms, representations and $L$-functions},
Proc. of Symp in Pure Math. vol. 33, Part 1 (1979), 189-202.

[BR1] A.I.Badulescu, D.Renard, Sur une conjecture de Tadi\'c, {\it Glasnik Matematicki 2004}.

[BR2] A.I.Badulescu, D.Renard, Zelevinsky involution and
Moeglin-Waldspurger algorithm for $GL(n,D)$, in {\it Functional
analysis IX} (Dubrovnik, 2005), 9-15., Various Publ. Ser. 48, Univ.
Aarhus, Aarhus, 2007.

[Ca] W.Casselman, {\it Introduction to the theory of admissible representations of reductive $p$-adic groups}
preprint.

[DKV] P.Deligne, D.Kazhdan, M.-F.Vignéras,
Représentations des algèbres centrales simples $p$-adiques, {\it
Représentations des groupes réductifs sur un corps local},
Hermann, Paris 1984.

[GJ] R.Godement, H.Jacquet, {\it Zeta functions of simple algebras},
Lecture Notes in Math. 260, Springer (1972).

[Fl1] Decomposition of representations into tensor products,
{\it Automorphic forms, representations and $L$-functions},
Proc. of Symp in Pure Math. vol. 33, Part 1 (1979), 179-185.

[Fl2] D.Flath, A comparison for the automorphic representations of GL(3) and its twisted forms, {\it Pacific J. Math.}
97 (1981), 373-402.

[Ja] Principal $L$-functions of the linear group,
{\it Automorphic forms, representations and $L$-functions},
Proc. of Symp in Pure Math. vol. 33, Part 2 (1979), 63-86.

[JL] H.Jacquet, R.P.Langlands, {\it Automorphic forms on
GL(2)}, Lecture Notes in Math. 114, Springer-Verlag (1970).

[JS] H.Jacquet, J.A.Shalika, On Euler products and the classification
of automorphic forms II, {\it Amer. J. Math.}  103 (1981), no. 4, 777-815.

[KS] C. D. Keys, F. Shahidi, Artin $L$-functions and normalization of
intertwining operators, {\it Ann. Sci. \'Ecole Norm. Sup. (4)} 21 (1988),
no. 1, 67-89.

[La] R.P.Langlands, On the notion of an automorphic representation,
{\it Automorphic forms, representations and $L$-functions},
Proceedings of Symposia in Pure Mathematics vol. 33, part. 1 (1979),  203-207.

[MW1] C. Moeglin, J.-L.Waldspurger, Sur l'involution de Zelevinski,
{\it J. Reine Angew. Math.}  372  (1986), 136-177.

[MW2] C.Moeglin, J.-L.Waldspurger, Le spectre r\'esiduel de ${\rm
  GL}(n)$,
{\it  Ann. Scient. \'Ec. Norm. Sup.} (4)  22  (1989),  no. 4,
605--674.

[Pi] R.S.Pierce, {\it Associative algebras}, Grad. Texts in Math., Vol 88, Springer-Verlag.

[P-S] I. Piatetski-Shapiro, Multiplicity one theorems,
{\it Automorphic forms, representations and $L$-functions},
Proceedings of Symposia in Pure Mathematics vol. 33 (1979), part. 1,
209-212.

[RV] D. Ramakrishnan, R. J. Valenza
{\it Fourier analysis on number fields}, Grad. Texts in Math., Vol. 186,
 Springer-Verlag, New York, 1999.

[Rod] F. Rodier,  Repr\'esentations de $GL(n,k)$ o\`u $k$ est un corps $p$-adique,
{\it S\'eminaire Bourbaki, Vol. 1981/1982}, Asterisque 92-93, Soc. Math. France, Paris, 1982
(201-218).

[Ro] J.Rogawski {\it Representations of $GL(n)$ and division algebras over a $p$-adic field},
Duke. Math. J. 50 (1983), 161-201.

[Sh] J.A.Shalika, The multiplicity one theorem for $GL_n$, {\it
Ann. of Math.} 100 (1974), 171-193.

[ScS] P.Schneider, U.Stuhler, Representation theory and sheaves on the
Bruhat-Tits building,  {\it Inst. Hautes \'Etudes Sci. Publ. Math.}  No. 85
(1997), 97--191.

[Ta1] M.Tadi\'c, Classification of unitary representations in irreducible
representations of a general linear group (non-Archimedean case),
{\it Ann. Scient. \'Ec. Norm. Sup.} 19 (1986), 335-382.

[Ta2] M.Tadi\'c, Induced representations of $GL(n;A)$ for a
$p$-adic division algebra $A$, {\it J. Reine angew. Math.} 405 (1990),
48-77.

[Ta3] M.Tadi\'c, An external approach to unitary representations, {\it Bulletin Amer. Math. Soc.} 28, No. 2 (1993),
215-252.

[Ta4] M.Tadi\'c, On characters of irreducible unitary representations of general
linear groups
{\it Abh. Math. Sem. Univ. Hamburg} 65
(1995), 341-363.

[Ta5] M.Tadi\'c, Correspondence on characters of irreducible unitary representations of $GL(n,\cc)$,

{\it Math. Ann. 305}
(1996), 419-438.

[Ta6] M.Tadi\'c, Unitarity and Jacquet-Langlands correspondences, a paraître dans {\it Pacific Journal}.

[Vi] M.-F. Vign\'eras, Correspondence betwen $GL_n$ and a division
algebra, {preprint}.

[Ze]  A.Zelevinsky, Induced representations of reductive
$p$-adic groups II, {\it Ann. Scient. \'Ec. Norm. Sup.} 13 (1980), 165-210.

\newpage

\appendix

\section{The Residual Spectrum of $GL_n$ over a
Division Algebra}

\centerline{\footnotesize{by Neven GRBAC}}

\bigskip

\subsection{Introduction}

In this Appendix the residual spectrum of $GL_n$ over a division
algebra is decomposed. The approach is the Langlands spectral
theory as explained in \cite{mwknjiga} and \cite{langlands}.
However, the results in the paper, obtained using the Arthur trace
formula of \cite{arthurclozel}, classify the entire discrete
spectrum of $GL_n$ over a division algebra. Hence, the problem
reduces to distinguishing the residual representations in the
discrete spectrum. This simplifies the application of the
Langlands spectral theory since it reduces the region of the
possible poles of the Eisenstein series to a cone well inside the
positive Weyl chamber. Having in mind the classification of the
discrete spectrum and the multiplicity one theorem, we obtain the
classification of the cuspidal spectrum as a consequence of the
decomposition of the residual spectrum. In fact, it turns out that
the only cuspidal representations are the basic cuspidal ones.

The idea of writing this Appendix was born during our stay at the
Erwin Schr\"{o}dinger Institute, Vienna, Austria in December 2006
and February 2007. I would like to thank Joachim Schwermer for his
kind invitation. My gratitude goes to Goran Mui\' c for many
useful conversations and constant help. I am grateful to Colette
M\oe glin for sharing her insight and advices on the normalization
of the standard intertwining operators. Also, I would like to
thank Marko Tadi\' c for the support and interest in my work. I
thank Ioan Badulescu for explaining his results and including this
Appendix to the paper. And finally, I would like to thank my wife
Tiki for bringing so much joy into my life.

\subsection{Normalization of Intertwining Operators}

Let $F$ be an algebraic number field (a global field of
characteristic zero) and $D$ a central division algebra of dimension
$d^2$ over $F$. Let $F_v$ denote the completion of $F$ at a place
$v$ and $\mathbb{A}$ the ring of ad\`{e}les of $F$. We use the
global notation of Sections 4 and 5. Let $G_r'$ be the inner form,
defined via $D$, of the split general linear group $G_{rd}=GL_{rd}$.
Let $V$ be a finite set of places where $D$ is non--split. As in the
paper, we assume that $D$ splits at all Archimedean places, i.e. $V$
consists only of non--Archimedean places.

Recall the description of the basic cuspidal automorphic
representations of $G_r'(\mathbb{A})$. Let $\rho$ be a cuspidal
automorphic representation of $G_{q}(\mathbb{A})$ and $s(\rho )$
the smallest positive integer such that the discrete spectrum
representation $\sigma\cong MW(\rho ,s(\rho))$ of
$G_{s(\rho)q}(\mathbb{A})$ is compatible at every place. Then,
$$
\sigma'\cong \g^{-1}(\sigma)\cong\otimes_v|\lj|_v(\sigma_v)
$$
is a basic cuspidal automorphic representation of
$G_r'(\mathbb{A})$. Observe that $\sigma_v'\cong\sigma_v$ at all
places $v\not\in V$. The goal of this Appendix is to show that all
cuspidal automorphic representations of $G_r'(\mathbb{A})$ are
obtained in this way. In fact, we show that all the remaining
representations in the discrete spectrum belong to the residual
spectrum and apply the multiplicity one theorem.

In the sequel we always assume that the cuspidal automorphic
representations are such that the poles of the attached Eisenstein
series and L--functions are real. There is no loss in generality
since this can be achieved simply twisting by the imaginary power
of the absolute value of the determinant. Hence, our assumption is
just a convenient choice of the coordinates. Furthermore, along
with the notation $\times$ for the parabolic induction, we use the
notation ${\Ind}_{M}^{G}$ when we want to point out the Levi
factor $M$ of the standard parabolic subgroup in $G$.

Consider first a cuspidal automorphic representation
$\sigma'\otimes \sigma'$ of the Levi factor $L'(\mathbb{A})\cong
G_r'(\mathbb{A})\times G_r'(\mathbb{A})$ of a maximal proper
standard parabolic subgroup in $G_{2r}'(\mathbb{A})$, where
$\sigma'$ is basic cuspidal as above. Let
$\underline{s}=(s_1,s_2)\in\mathfrak{a}_{L',\mathbb{C}}$ and $w$
the unique nontrivial Weyl group element such that $wL'w^{-1}=L'$.

\begin{lemme}\label{lema:split}
Let $v\not\in V$ be a split place. The normalizing factor for the
standard intertwining operator
$$
A((s_1,s_2),\sigma_v\otimes\sigma_v,w)
$$
acting on the induced representation
$$
{\Ind}_{G_{rd}(F_v)\times
G_{rd}(F_v)}^{G_{2rd}(F_v)}\left(\nu^{s_1}\sigma_v\otimes\nu^{s_2}\sigma_v\right)
$$
is given by
\begin{equation}\label{eq:normfactorsplit}
r((s_1,s_2),\sigma_v\otimes\sigma_v,w)=\frac{\prod_{j=1}^{s(\rho)}L(s_1-s_2-s(\rho)+j,\rho_v\times\widetilde{\rho}_v)}
{\prod_{j=1}^{s(\rho)}L(s_1-s_2+j,\rho_v\times\widetilde{\rho}_v)\cdot
\varepsilon(s_1-s_2,\sigma_v\times\widetilde{\sigma}_v',\psi_v)},
\end{equation}
where the L--functions and $\varepsilon$--factors are the local
Rankin--Selberg ones of pairs. Then, the normalized intertwining
operator $N((s_1,s_2),\sigma_v\otimes\sigma_v,w)$, defined by
$$
A((s_1,s_2),\sigma_v\otimes\sigma_v,w)=r((s_1,s_2),\sigma_v\otimes\sigma_v,w)N((s_1,s_2),\sigma_v\otimes\sigma_v,w),
$$
is holomorphic and non--vanishing for $Re(s_1-s_2)\geq s(\rho)$.
\end{lemme}

\begin{proof}
This Lemma is a weaker form of Lemma I.10 of \cite{mwGLn} where
the holomorphy and non--vanishing is proved in a certain region
slightly bigger than the closure of the positive Weyl chamber for
any unitary representation. We just show that the normalizing
factor defined in \cite{mwGLn} is the same as here.

By \cite{mwGLn},
\begin{equation}\label{eq:normfaktorMW}
r((s_1,s_2),\sigma_v\otimes\sigma_v,w)=\frac{L(s_1-s_2,\sigma_v\times
\widetilde{\sigma}_v)}{L(1+s_1-s_2,\sigma_v\times
\widetilde{\sigma}_v)\varepsilon(s_1-s_2,\sigma_v\times
\widetilde{\sigma}_v,\psi_v)}.
\end{equation}
But, $\sigma_v$ is a quotient of the induced representation
$$
\nu^{\frac{s(\rho)-1}{2}}\rho_v\times\nu^{\frac{s(\rho)-3}{2}}\rho_v\times
\ldots \times \nu^{-\frac{s(\rho)-1}{2}}\rho_v,
$$
where $\rho_v$, being unitary and generic as the local component
at $v$ of a cuspidal automorphic representation $\rho$, is a fully
induced representation of the form
$$
\nu^{e_{1,v}}\delta_{1,v}\times\nu^{e_{2,v}}\delta_{2,v}\times
\ldots \times\nu^{e_{m_v,v}}\delta_{m_v,v}
$$
with $e_{i,v}$ real, $|e_{i,v}|<1/2$ and
$\delta_{i,v}\in\mathcal{D}^u$. We may arrange the indices in such
a way that $e_{1,v}\geq e_{2,v}\geq\ldots \geq e_{m_v,v}$.

This shows that $\sigma_v$ is the Langlands quotient and we can
apply the formulas for the Rankin--Selberg L--function and
$\varepsilon$--factor of the Langlands quotient. Having in mind
that $\rho_v$ is fully induced, we obtain
\begin{equation}\label{eq:LfjaRS}
L(s,\sigma_v\times\widetilde{\sigma}_v)=L(s,\rho_v\times\widetilde{\rho}_v)^{s(\rho)}
\prod_{j=1}^{s(\rho)-1}L(s+s(\rho)-j,\rho_v\times\widetilde{\rho}_v)^j
L(s-s(\rho)+j,\rho_v\times\widetilde{\rho}_v)^j
\end{equation}
and the $\varepsilon$--factor is of the same form, but since it
has no zeroes nor poles we do not need to refine its form.
Inserting the formula for the L--function into Equation
(\ref{eq:normfaktorMW}) gives after cancellation the normalizing
factor (\ref{eq:normfactorsplit}).
\end{proof}


\begin{lemme}\label{lema:nonsplit}
Let $v\in V$ be a non--split place. Then the standard intertwining
operator
$$
A((s_1,s_2),\sigma_v'\otimes\sigma_v',w)
$$
is holomorphic and non--vanishing for $Re(s_1-s_2)\geq s(\rho)$.
\end{lemme}

\begin{proof}
Sections 3.2, 3.3 and 3.5 give rather precise form of the local
component $\sigma_v'$ of a basic cuspidal automorphic
representation of $GL_r'(\mathbb{A})$. By Section 3.5, it is a
fully induced representation of the form
$$
\sigma_v'\cong
\nu^{e_{1,v}}|\mathbf{LJ}|_v\left(u(\delta_{1,v},s(\rho))\right)\times\ldots\times
\nu^{e_{m_v,v}}|\mathbf{LJ}|_v\left(u(\delta_{m_v,v},s(\rho))\right),
$$
where $e_{i,v}$ are real, $|e_{i,v}|<1/2$ and
$\delta_{i,v}\in\mathcal{D}^u$. More precisely, $e_{i,v}$ and
$\delta_{i,v}$ are defined by
$$
\rho_v\cong\nu^{e_{1,v}}\delta_{1,v}\times\ldots\times
\nu^{e_{m_v,v}}\delta_{m_v,v}.
$$
The precise formula for
$|\mathbf{LJ}|_v\left(u(\delta_{i,v},s(\rho))\right)$ is given in
Proposition 3.7 and Equation (3.8). If $\delta_{i,v}$ is compatible,
then
$$
|\mathbf{LJ}|_v\left(u(\delta_{i,v},s(\rho))\right)=
\overline{u}(\delta_{i,v}',s(\rho)),
$$
and the highest exponent of $\nu$ appearing in the corresponding
standard module is $\frac{s(\rho)-1}{2}$. If $\delta_{i,v}$ is not
compatible, then, by the choice of $s(\rho)$, we have
$$
|\mathbf{LJ}|_v\left(u(\delta_{i,v},s(\rho))\right)= \prod_{i=1}^b
\nu^{i-\frac{b+1}{2}}u'(\delta_{i,+,v}',s(\rho)/s(\delta_{i,v}))\times
\prod_{j=1}^{s(\delta_{i,v})-b}\nu^{j-\frac{s(\delta_{i,v})-b+1}{2}}u'(\delta_{i,-,v}',s(\rho)/s(\delta_{i,v})),
$$
where $\delta_{i,\pm,v}'\in\mathcal{D}'^u$ are certain unitary
discrete series representations. See Section 3.3 for the
unexplained notation. In this case the highest exponent of $\nu$
appearing among the standard modules is either
$$
\frac{b-1}{2}+s(\delta_{i,v})\frac{s(\rho)/s(\delta_{i,v})-1}{2}<\frac{s(\rho)-1}{2}
$$
or
$$
\frac{s(\delta_{i,v})-b-1}{2}+s(\delta_{i,v})\frac{s(\rho)/s(\delta_{i,v})-1}{2}\leq\frac{s(\rho)-1}{2},
$$
where the upper bounds are obtained using the fact that $0\leq
b<s(\delta_{i,v})$ (see Section 3.3).

The description of $\sigma_v'$ shows that the induced
representation
$$
\nu^{s_1}\sigma_v'\times\nu^{s_2}\sigma_v'
$$
is a product of possibly twisted representations of the form
$\overline{u}(\cdot )$ and $u'(\cdot )$ which are the Langlands
quotients of the standard module induced from a discrete series
representation. In other words there is a unitary discrete series
representation $\delta_v'$ of the appropriate Levi factor
$L_0'(F_v)$ of $G_{2r}'(F_v)$ and
$\underline{s}\in\mathfrak{a}_{L_0',\mathbb{C}}$ such that, by the
Langlands classification, the standard intertwining operator
$$
A(\underline{s},\delta_v',w_0):
{\Ind}_{L_0'(k_v)}^{G_{2r}'(k_v)}(\underline{s},\delta_v')\to
{\Ind}_{w_0(L_0')(k_v)}^{G_{2r}'(k_v)}(w_0(\underline{s}),w_0(\delta_v'))
$$
is holomorphic and its image is the induced representation
$\nu^{s_1}\sigma_v'\times\nu^{s_2}\sigma_v'$. Therefore, the
standard intertwining operator
$A((s_1,s_2),\sigma_v'\otimes\sigma_v',w)$ fits into the
commutative diagram
$$
{\Ind}_{L_0'(k_v)}^{G_{2r}'(k_v)}(\underline{s},\delta_v')\quad
{}^{\underrightarrow{\quad A({\underline{s}},\delta_v',w_0)\quad}}
\quad \nu^{s_1}\sigma_v'\times\nu^{s_2}\sigma_v'
$$
$$
{\quad\quad}^{A(\underline{s},\delta_v',ww_0)}\downarrow
\quad\quad\quad\quad\quad\quad\quad\quad\quad\quad\quad\quad\quad\quad\quad
\downarrow {}^{A((s_1,s_2),\sigma_v'\otimes\sigma_v',w)}
$$
$$
{\Ind}_{ww_0(L_0')(k_v)}^{G_{2r}'(k_v)}(ww_0(\underline{s}),ww_0(\delta_v'))\quad
\hookleftarrow \quad\nu^{s_2}\sigma_v'\times\nu^{s_1}\sigma_v',
$$
where the upper horizontal arrow is surjective. The diagram
implies the Lemma if we prove that, for $Re(s_1-s_2)\geq s(\rho)$,
the left vertical arrow is holomorphic and non--vanishing.

By the Langlands classification it suffices to check that the real
parts of all the differences between exponents of $\nu$ appearing in
the parts of $I(\underline{s},\delta_v')$ corresponding to
$\nu^{s_1}\sigma_v'$ and $\nu^{s_2}\sigma_v'$ are strictly positive.
However, we already checked that the highest exponent appearing
among the standard modules in the expressions for
$|\mathbf{LJ}|_v(u(\delta_{i,v},s(\rho)))$ is at most
$\frac{s(\rho)-1}{2}$. Therefore, in the worst case we obtain the
difference
$$
Re(s_1-s_2)+e_{i,v}-e_{j,v}-2\cdot\frac{s(\rho)-1}{2}>0
$$
since $e_{i,v}-e_{j,v}>-1$.
\end{proof}

\begin{rem}
The proof of the previous Lemma follows the idea of the proof of
Lemma I.8 of \cite{mwGLn}. Since the results of this paper based
on the trace formula reduce the question of determining the
residual spectrum to the point $Re(s_1-s_2)=s(\rho)$ and give
bounds on the exponents of the local component at a non--split
place of a cuspidal automorphic representation of an inner form,
we do not require the full power of Lemma I.8, and hence the proof
becomes simpler. However, its analogue for inner forms could have
been obtained using first the transfer of the Plancherel measure
for discrete series representations (see \cite{muicsavin}) to
define the normalization using L--functions for the split group.
For the classical hermitian quaternionic groups we used this
technique to obtain the parts of the residual spectra in
\cite{jaDubrovnik}, \cite{jaSteinberg}, \cite{jaSO8},
\cite{jaSp8}.
\end{rem}

\begin{cor}\label{kor:global}
The normalizing factor for the global standard intertwining
operator
$$
A((s_1,s_2),\sigma'\otimes\sigma',w)
$$
acting on the induced representation
$$
{\Ind}_{L'(\mathbb{A})}^{G'_{2r}(\mathbb{A})}\left(\nu^{s_1}\sigma'\otimes\nu^{s_2}\sigma'\right)
$$
is given by
\begin{equation}\label{eq:normfactorglobal}
r((s_1,s_2),\sigma'\otimes\sigma',w)=\frac{\prod_{j=1}^{s(\rho)}L_V(s_1-s_2-s(\rho)+j,\rho\times\widetilde{\rho})}
{\prod_{j=1}^{s(\rho)}L_V(s_1-s_2+j,\rho\times\widetilde{\rho})\cdot
\varepsilon_V(s_1-s_2,\sigma'\times\widetilde{\sigma}')},
\end{equation}
where the L--functions and $\varepsilon$--factors are the partial
Rankin--Selberg ones with respect to the finite set $V$ of
non--split places of $D$. Then, the normalized intertwining
operator $N((s_1,s_2),\sigma'\otimes\sigma',w)$ defined by
$$
A((s_1,s_2),\sigma'\otimes\sigma',w)=r((s_1,s_2),\sigma'\otimes\sigma',w)N((s_1,s_2),\sigma'\otimes\sigma',w)
$$
is holomorphic and non--vanishing for $Re(s_1-s_2)\geq s(\rho)$.
Moreover, the only pole of the standard intertwining operator
$A((s_1,s_2),\sigma'\otimes\sigma',w)$ in the region
$Re(s_1-s_2)\geq s(\rho)$ is at $s_1-s_2=s(\rho)$ and it is
simple.
\end{cor}

\begin{proof}
The global normalizing factor is obtained as a product over all
places of the local ones. Note that, for our purposes, at a
non--split places the normalizing factor is taken to be trivial.
Then the holomorphy and non--vanishing of the normalized
intertwining operator in the region $Re(s_1-s_2)\geq s(\rho)$
follows from the local results of the previous two Lemmas.

The analytic properties of the Rankin--Selberg L--functions are
well--known. The global Rankin--Selberg L--function
$L(z,\rho\times\widetilde{\rho})$ has the only poles at $z=0$ and
$z=1$ and they are both simple. It has no zeroes for $Re(z)\geq
1$. Writing $\rho_v$ at a non--split place $v\in V$ as a fully
induced representation from the discrete series representation as
in the proof of the previous Lemma shows that the local
Rankin--Selberg L--function equals
$$
L(z,\rho_v\times\widetilde{\rho_v})=\prod_{i,j=1}^{m_v}L(z+e_{i,v}-e_{j,v},\delta_{i,v}\times\widetilde{\delta}_{j,v}).
$$
Since the local L--functions attached to unitary discrete series
representations are holomorphic in the strict right half--plane,
and $e_{i,v}-e_{j,v}>-1$, the L--function
$L(z,\rho_v\times\widetilde{\rho_v})$ is holomorphic for
$Re(z)\geq 1$. Local L--functions have no zeroes.

Therefore, the partial L--function
$L_V(z,\rho\times\widetilde{\rho})$ is holomorphic for $Re(z)\geq
1$ except for a simple pole at $z=1$. It has no zeroes for
$Re(z)\geq 1$. The $\varepsilon$--factor has neither zeroes nor
poles. Since for $Re(s_1-s_2)\geq s(\rho)$ real parts of all the
arguments of the L--functions in the global normalizing factor
(\ref{eq:normfactorglobal}), except $Re(s_1-s_2-s(\rho)+1)\geq 1$,
are strictly greater than one, it has no zeroes and the only pole
occurs for $s_1-s_2=s(\rho)$. Since the normalized intertwining
operator is holomorphic and non--vanishing for $Re(s_1-s_2)\geq
s(\rho)$, it turns out that the only pole in the region
$Re(s_1-s_2)\geq s(\rho)$ of the global standard intertwining
operator is at $s_1-s_2=s(\rho)$ and it is simple.
\end{proof}

\subsection{Poles of Eisenstein Series}

Let $\sigma'$ be as above and $k>1$ an integer. Let $\pi'\cong
\sigma'\otimes\ldots\otimes\sigma'$ be a cuspidal automorphic
representation of the Levi factor $M'(\mathbb{A})\cong
G_r'(\mathbb{A})\times\ldots\times G_r'(\mathbb{A})$ of a standard
parabolic subgroup of $G_{kr}'(\mathbb{A})$, with $k$ copies of
$G_r'(\mathbb{A})$ and $\sigma'$ in the products. We fix an
isomorphism $\mathfrak{a}_{M',\mathbb{C}}^\ast\cong \mathbb{C}^k$
using the absolute value of the reduced norm of the determinant at
each copy of $G_r'$ and denote its elements by
$\underline{s}=(s_1,s_2,\ldots
,s_k)\in\mathfrak{a}_{M',\mathbb{C}}^\ast$. By the results of the
paper, the study of the residual spectrum is reduced to the point
$$
\underline{s}_0=\left(\frac{s(\rho)(k-1)}{2},\frac{s(\rho)(k-3)}{2},\ldots
,-\frac{s(\rho)(k-1)}{2}\right),
$$
i.e. we have to prove that the unique discrete series constituent
of the induced representation
$$
{\Ind}_{M'(\mathbb{A})}^{G_{kr}'(\mathbb{A})}\left(\underline{s}_0,\pi'\right)=
\nu^{\frac{s(\rho)(k-1)}{2}}\sigma'\times\nu^{\frac{s(\rho)(k-3)}{2}}\sigma'\times
\ldots\times\nu^{-\frac{s(\rho)(k-1)}{2}}\sigma',
$$
which is denoted in the paper by $MW'(\sigma',k)$, is in the
residual spectrum. Of course, the case $k=1$ is excluded since it
gives just the (basic) cuspidal representation $\sigma'$.

\begin{lemme}\label{lema:eisenstein}
Let
$$
E(\underline{s},g;\pi',f_{\underline{s}})
$$
be the Eisenstein series attached to a 'good' (in a sense of
Sections II.1.1 and II.1.2 of \cite{mwknjiga}) section
$f_{\underline{s}}$ of the above induced representation from a
cuspidal automorphic representation $\pi'$. Then, its only pole in
the region $Re(s_i-s_{i+1})\geq s(\rho)$, for $i=1,\ldots ,k-1$,
is at $\underline{s}_0$ and it is simple. The constant term map
gives rise
to an isomorphism between the space of automorphic forms
$\mathcal{A}(\sigma',k)$ spanned by the iterated residue at
$\underline{s}_0$ of the Eisenstein series and the irreducible
image $MW'(\sigma',k)$ of the normalized intertwining operator
$$
N(\underline{s}_0,\pi',w_l),
$$
where $w_l$ is the longest among Weyl group elements $w$ such that
$wM'w^{-1}\cong M'$.
\end{lemme}

\begin{proof}
By the general theory of the Eisenstein series explained in
Section V.3.16 of \cite{mwknjiga}, its poles coincide with the
poles of its constant term along the standard parabolic subgroup
with the Levi factor $M'$ which equals the sum of the standard
intertwining operators
$$
E_0(\underline{s},g;\pi',f_{\underline{s}})=\sum_{w\in W(M')}
A(\underline{s},\pi',w)f_{\underline{s}}(g),
$$
where $W(M')$ is the set of the Weyl group elements such that
$wM'w^{-1}\cong M'$. Hence, the poles of the Eisenstein series are
determined by the poles of the standard intertwining operators.

By Corollary \ref{kor:global}, in the region $Re(s_i-s_{i+1})\geq
s(\rho)$, for $i=1,\ldots ,k-1$, the only possibility for the pole
is at $\underline{s}_0$. However, it indeed occurs only for the
intertwining operators corresponding to the Weyl group element
inverting the order of any two successive indices, i.e. the
longest element $w_l$ in $W(M')$. Since the iterated pole is
simple in every iteration, the iterated residue of the constant
term, up to a non--zero constant, equals the normalized
intertwining operator
$$
N(\underline{s}_0,\pi',w_l),
$$
as claimed.

The irreducibility of its image follows from the uniqueness of the
discrete series constituent in the considered induced
representation obtained in Proposition 5.6(a). The square
integrability follows from the Langlands criterion (Section I.4.11
of \cite{mwknjiga}).
\end{proof}

\begin{rem}
The proof of the Lemma shows that $MW'(\sigma',k)$, for $k>1$, is
at every place an irreducible quotient of the corresponding
induced representation.
\end{rem}

\begin{theo}\label{tm:residualspec}
The residual spectrum $L_{\mathrm{res}}^2(G_n')$ of an inner form
$G_n'(\mathbb{A})$ of the split general linear group decomposes
into a Hilbert space direct sum
$$
L_{\mathrm{res}}^2(G_n')\cong\bigoplus_{\left.%
\begin{array}{c}
    r|n \\
    1 \leq r<n \\
\end{array}%
\right.}\bigoplus_{\left.%
\begin{array}{c}
    \sigma'\in DS_r' \\
    \mathrm{(basic)}\hbox{ }
\mathrm{cuspidal} \\
\end{array}%
\right.} \mathcal{A}(\sigma',n/r),
$$
where $\mathcal{A}(\sigma',n/r)\cong MW'(\sigma',n/r)$ are the
spaces of automorphic forms obtained in the previous Lemma.
\end{theo}

\begin{proof}
The results of Section 5 classify the discrete spectrum $DS_n'$ of
the inner form $G_n'(\mathbb{A})$ using the trace formula. The
basic cuspidal representations are proved to be cuspidal
automorphic. Hence, it remains to show that the representations of
the form $MW'(\sigma',k)$, for $k>1$ and a basic cuspidal
representation $\sigma'$, are in the residual spectrum. However,
this is precisely the content of the previous Lemma
\ref{lema:eisenstein}.
\end{proof}

\begin{cor}\label{cor:cusp}
The cuspidal spectrum of an inner form $G_n'(\mathbb{A})$ consists
of the basic cuspidal automorphic representations.
\end{cor}

\begin{proof}
Theorem \ref{tm:residualspec} shows that in the discrete spectrum $DS_n'$
of an inner form $G_n'(\mathbb{A})$ obtained in Section 5 all the
representations not being basic cuspidal belong to the residual
spectrum. Hence, the multiplicity one of Theorem \ref{correspondence} for inner forms implies the
Corollary.
\end{proof}


\begin{thebibliography}{99}

\bibitem[AC]{arthurclozel}
J. Arthur, L. Clozel, {\it Simple Algebras, Base Change, and the
Advanced Theory of the Trace Formula}, Ann. of Math. Studies {\bf 120},
Princeton Univ. Press, 1989

\bibitem[Gr1]{jaDubrovnik}
N. Grbac, Correspondence between the Residual Spectra of Rank Two Split Classical Groups and their Inner Forms,
{\it Functional
analysis IX} (Dubrovnik, 2005), 44--57, Various Publ. Ser. {\bf 48}, Univ.
Aarhus, Aarhus, 2007

\bibitem[Gr2]{jaSteinberg}
N. Grbac, On a Relation between Residual Spectra of Split Classical Groups and their Inner Forms,
{\it Canad. J. Math.}, to appear

\bibitem[Gr3]{jaSO8}
N. Grbac, On the Residual Spectrum of Hermitian Quaternionic Inner Form of $SO_8$, preprint

\bibitem[Gr4]{jaSp8}
N. Grbac, The Residual Spectrum of an Inner Form of $Sp_8$ Supported in the Minimal Parabolic Subgroup,
preprint

\bibitem[La2]{langlands}
R.P. Langlands, {\it On the Functional Equations Satisfied
by Eisenstein series}, Lecture Notes in Math. {\bf 544},
Springer--Verlag, 1976

\bibitem[MW2]{mwGLn}
C. M\oe glin, J.--L. Waldspurger, Le spectre r\' esiduel de
$GL(n)$, {\it Ann. scient. \' Ec. Norm. Sup.} {\bf 22} (1989),
605--674

\bibitem[MW3]{mwknjiga}
C. M\oe glin, J.--L. Waldspurger, {\it Spectral
Decomposition and Eisenstein Series}, Cambridge Tracts in Math.
{\bf 113}, Cambridge University Press, 1995

\bibitem[MS]{muicsavin}
G. Mui\' c, G. Savin, Complementary Series for Hermitian
Quaternionic Groups, {\it Canad. Math. Bull.} {\bf 43} (2000),
90--99


\end{thebibliography}
\end{document}